\documentclass[12pt, reqno]{amsart}
\usepackage{amsmath, amsthm, amscd, amsfonts, amssymb, graphicx, color}
\usepackage[bookmarksnumbered, colorlinks, plainpages]{hyperref}
\hypersetup{colorlinks=true,linkcolor=red, anchorcolor=green, citecolor=cyan, urlcolor=red, filecolor=magenta, pdftoolbar=true}

\newtheorem{theorem}{Theorem}[section]
\newtheorem{lemma}[theorem]{Lemma}
\newtheorem{proposition}[theorem]{Proposition}
\newtheorem{corollary}[theorem]{Corollary}

\theoremstyle{definition}
\newtheorem{definition}[theorem]{Definition}
\newtheorem{example}[theorem]{Example}

\theoremstyle{remark}

\numberwithin{equation}{section}

\newcommand\C{\mathbb C}
\newcommand\D{\mathbb D}

\newcommand\rtkmu{R^t(K, \mu)}

\newcommand\rtkm{R^t(K, \mu)}

\newcommand\rikmu{R^\infty(K,\mu)}

\newcommand\limu{L^\infty(\mu)}
\renewcommand\i{\infty}
\newcommand\area{{\frak m}}
\newcommand\CT{{\mathcal C}}

\newcommand{\quotes}[1]{``#1''}

\begin{document}

\setcounter{page}{1}

\title[The Commutant of Multiplication by z]{The Commutant of Multiplication by z on the Closure of Rational Functions in $L^t(\mu)$}

\author[L. Yang]{Liming Yang$^1$}

\address{$^1$Department of Mathematics, Virginia Polytechnic and State University, Blacksburg, VA 24061.}
\email{\textcolor[rgb]{0.00,0.00,0.84}{yliming@vt.edu}}




\begin{abstract} For a compact set $K\subset \mathbb C,$ a finite positive Borel measure $\mu$ on $K,$ and $1 \le t < \i,$ let $\text{Rat}(K)$ be the set of rational functions with poles off $K$ and let $R^t(K, \mu)$ be the closure of $\text{Rat}(K)$ in $L^t(\mu).$ For a bounded Borel subset $\mathcal D\subset \mathbb C,$ let $\area_{\mathcal D}$ denote the area (Lebesgue) measure restricted to $\mathcal D$ and let $H^\i (\mathcal D)$ be the weak-star closed sub-algebra of $L^\i(\area_{\mathcal D})$ spanned by $f,$ bounded and analytic on $\mathbb C\setminus E_f$ for some compact subset $E_f \subset \mathbb C\setminus \mathcal D.$ We show that if $R^t(K, \mu)$ contains no non-trivial direct $L^t$ summands, then there exists a Borel subset $\mathcal R \subset K$ whose closure contains the support of $\mu$ and there exists an isometric isomorphism and a weak-star homeomorphism $\rho$ from $R^t(K, \mu) \cap L^\infty(\mu)$ onto $H^\infty(\mathcal R)$ such that $\rho(r) = r$ for all $r\in\text{Rat}(K).$ Consequently, we obtain some structural decomposition theorems for $\rtkmu$.
\end{abstract} \maketitle

\section{\textbf{Introduction}}

For a Borel subset $E$ of the complex plane $\C,$ let $M_0(E)$ denote the set of finite complex-valued Borel measures that are compactly supported in $E$ and let $M_0^+(E)$ be the set of positive measures in $M_0(E).$ For $\nu\in M_0(\C),$ we use $\text{spt}(\nu)$ for the support of $\nu,$ $\nu_B$ for $\nu$ restricted to $B\subset \C,$ and $\CT(\nu)$ for the Cauchy transform of $\nu$ (see section 2 for its definition).

For a compact subset $K\subset \C,$ $\mu\in M_0^+(K),$ $1\leq t < \infty$, and $\frac 1t + \frac 1s = 1,$ functions in 
$\mbox{Rat}(K) := \{q:\mbox{$q$ is a rational function 
with poles off $K$}\}$
are members of $L^t(\mu)$. Let $R^t(K, \mu)$ denote the closure of $\mbox{Rat}(K)$ in $L^t(\mu)$ norm. We say $g\perp \rtkmu$ (or $g\in \rtkmu ^\perp$) if $g\in L^s(\mu)$ and $\int rgd\mu = 0$ for all $r\in \text{Rat}(K).$
Let $S_\mu$ denote the multiplication by $z$ on $R^t(K, \mu).$
The operator $S_\mu$ is pure if $R^t(K, \mu)$ does not have a non-trivial direct $L^t$ summand.
  The commutant of $S_\mu$ is denoted
 \[
 \ R^{t,\i}(K, \mu) = R^t(K, \mu) \cap L^\infty(\mu).
 \] 
 
 For a bounded Borel subset $\mathcal D\subset \mathbb C,$ let 
 $H(\mathcal D)$ be the set of functions $f,$ where $f$ is a bounded and analytic function on $\C \setminus E_f$ for some compact subset $E_f \subset \C \setminus \mathcal D$ (depending on $f$).
 Let $H^\i (\mathcal D)$ denote the weak-star closure of $H(\mathcal D)$ in $L^\i(\area_{\mathcal D}),$ where $\area$ stands for the area (Lebesgue) measure on $\C$. If $\mathcal D$ is a bounded open subset, then $H^\infty(\mathcal D)$ is actually the algebra of bounded and analytic functions on $\mathcal D.$ In general, $\mathcal D$ may not have interior (i.e. a Swiss cheese set, see Example \ref{FRExample}).
 We use $\bar E$ for the closure of a subset $E.$ 
 
 Our main theorem in this paper is the following.
 
 \begin{theorem}\label{mainThm}
 Let $\mu\in M_0^+(K)$ for a compact set $K\subset \C$.
 If $1\le t < \infty$ and $S_\mu$ on $R^t(K, \mu)$ is pure, 
then there is a Borel subset $\mathcal R \subset K$ with $\text{spt}
\mu \subset \overline{\mathcal R}$ and there is an isometric isomorphism and a weak-star homeomorphism $\rho$ from $R^{t,\i}(K, \mu)$ onto $H^\infty(\mathcal R)$ such that 
\newline
(1) $\rho(r) = r$ for $r\in\text{Rat}(K),$	
\newline
(2) $\CT(g\mu)(z) = 0,~z\in \C \setminus \mathcal R,~\area-a.a.$ for $g\perp \rtkmu,$ and
\newline
(3) $\rho(f)(z)\CT(g\mu)(z) = \CT(fg\mu)(z),~ \area-a.a.$ for $f\in R^{t,\i}(K, \mu)$ and $g\perp \rtkmu.$
 \end{theorem}
 
As applications, we obtain some structural decomposition theorems (Theorem \ref{DecompTheoremRT} and Theorem \ref{thmA}) for  $\rtkmu.$ Consequently, Corollary \ref{thmB} extends the results of \cite{thomson}, \cite{ce93}, and \cite{b08} when the boundary of $K$ is not too wild.

We briefly outline how we define the map $\rho$ that associates to each function in $\rtkmu$ a point function on $\mathcal R.$ For $g\in\rtkmu ^\perp$ and $r\in \text{Rat}(K),$ we have $\frac{r(z)-r(\lambda)}{z-\lambda} \in \rtkmu$ for $\lambda\in K,$ so $\int \frac{r(z)-r(\lambda)}{z-\lambda} g(z)d\mu(z) = 0.$ By Corollary \ref{ZeroAC},
\begin{eqnarray}\label{OutlineEq1}
\ r(\lambda)\CT (g\mu)(\lambda) = \CT (rg\mu)(\lambda),~ \gamma-a.a.,
\end{eqnarray} 
where $\gamma$ stands for analytic capacity (see section 2 for its definition) and we use $\gamma-a.a.$ for a property that holds everywhere except possibly on a set of analytic capacity zero. Let $\{g_j\} \subset \rtkmu ^\perp$ be a $L^s(\mu)$ norm dense subset of $\rtkmu ^\perp.$ Let $\mathcal N$ denote the collection of $\lambda\in \C$ satisfying: there exists $j$ such that $\mathcal C(g_j\mu)(\lambda) = \lim_{\epsilon \rightarrow 0} \mathcal C_\epsilon(g_j\mu)(\lambda)$ (principal value, see section 2) exists and $\mathcal C(g_j\mu)(\lambda) \ne 0.$ We fix $f\in \rtkmu.$ Choose a sequence $\{r_n\} \subset \text{Rat}(K)$ such that 
\[
\ \|r_n - f\|_{L^t(\mu)}\rightarrow 0 \text{ and } r_n(z) \rightarrow f(z),~\mu-a.a..
\]
As an application of Tolsa's Theorem ( see Theorem \ref{TolsaTheorem} and Lemma \ref{ConvergeLemma} for details), we infer that there exists a subsequence $\{r_{n_k}\}$ such that for all $j\ge 1,$
\[
\ \CT (r_{n_k}g_j\mu)(\lambda) \rightarrow \CT (fg_j\mu)(\lambda), ~\gamma-a.a.\text{ as }k \rightarrow \i.
\] 
Using \eqref{OutlineEq1}, we conclude that $r_{n_k}$ converges to a function, denoted $\rho(f),$ on $\mathcal N,~~\gamma-a.a..$ Then $\rho(f)$ satisfies 
\[
\ \rho(f)(\lambda)\CT (g_j\mu)(\lambda) = \CT (fg_j\mu)(\lambda), ~\gamma-a.a.\text{ for } j\ge 1.
\]
We will prove that the set 
$\mathcal R$ in Theorem \ref{mainThm} consists of $\lambda\in \mathcal N$ such that there exists an integer $N_\lambda \ge 1$ satisfying 
\[
\ \lim_{\delta\rightarrow 0} \dfrac{\gamma(\D(\lambda, \delta) \cap \{\max_{1\le j \le N_\lambda}|\CT(g_j\mu)| \le N_\lambda^{-1} \})}{\delta} = 0, 
\]
where $\D(\lambda, \delta) := \{z:~|z - \lambda| < \delta\}$ (see Theorem \ref{FCharacterization}).

For $t = \i,$ let $\rikmu$ be the weak-star closure of $\mbox{Rat}(K)$ in $\limu.$ Chaumat's Theorem \cite{ch74} states that if $\rikmu$ contains no non-trivial direct $L^\i$ summands, then there exists a subset $E\subset K$ such that the identity map $\rho(r) = r$ for $r\in \mbox{Rat}(K)$ extends an isometric isomorphism and a weak-star homeomorphism from $\rikmu$ onto $R^\i (\overline E, \area_E)$ (also see \cite[Chaumat's Theorem on page 288]{conway}). 
The envelope $E$ consists of points $\lambda \in  K$ satisfying: there exists $g_\lambda \perp \rikmu$ such that $\int \frac{|g_\lambda|}{|z - \lambda|} d\mu<\i$ and $\CT (g_\lambda\mu)(\lambda) \ne 0.$ However, the concept of $E$ can not be used for studying $R^{t,\i}(K,\mu)$ as there are points $\lambda \in \mathcal N$ and $j$ such that $\int \frac{|g_j|}{|z - \lambda|} d\mu = \i,$ the principal value of $\CT (g_j\mu)(\lambda)$ exists, and $\CT (g_j\mu)(\lambda) \ne 0.$ We will see that those points are important for the structure of $R^{t,\i}(K,\mu).$ 

Tolsa's Theorems (see \cite{Tol02}, \cite{Tol03}, and \cite{Tol14}) on analytic capacity and Cauchy transform provide us necessary tools in our approach. In section 2, we review some results of analytic capacity and Cauchy transform that are needed in our analysis. We define in details the map $\rho$ that maps each function in $\rtkmu$ to a point function on $\mathcal N,~\gamma-a.a.$ (see Lemma \ref{RhoExistLemma}). We also discuss some basic properties of $\rho.$ In section 3, 
we define the non-removable boundary $\mathcal F$ for an arbitrary compact subset $K$ and $\mu\in M_0^+(K)$. Intuitively, the set $\mathcal F$ splits into three sets, $\mathcal F_0$ and $\mathcal F_+ \cup \mathcal F_-$ such that
(1) Cauchy transforms of annihilating measures $g\mu$ ($g\perp R^t(K, \mu)$) are zero on $\mathcal F_0$ and (2) Cauchy transforms of annihilating measures $g\mu$ have zero \quotes{one side non-tangential limits} on $\mathcal F_+ \cup \mathcal F_-.$
The removable set is defined by $\mathcal R = \C \setminus \mathcal F.$ 
The set $\mathcal N$ is decomposed into $\mathcal R \cup \mathcal F_+ \cup \mathcal F_-.$ Theorem \ref{FACDensityThm} proves that there exists a subset $\mathcal Q$ with $\gamma(\mathcal Q)=0$ such that every $\lambda \in \mathcal R\setminus \mathcal Q$ satisfies $\lim_{\delta\rightarrow 0}\frac{\gamma(\D(\lambda,\delta) \setminus \mathcal R)}{\delta} = 0$ (full analytic capacity for $\mathcal R$), that is, the set $\mathcal R$ is a \quotes{nearly open} subset. 
Theorem \ref{HDAlgTheorem} characterizes $H^\i(\mathcal R),$ the algebra of \quotes{bounded and analytic} functions on a \quotes{nearly open} subset. The proof is basically given by \cite[Theorem 4.3]{y23}. Using Theorem \ref{HDAlgTheorem} and Lemma \ref{BBFRLambda}, we prove that $\rho$ is surjective (Lemma \ref{RhoOntoLemma}) in section 5. In section 6, combining Lemma \ref{MLemma1} with Theorem \ref{HDAlgTheorem}, we conclude that $\rho$ maps $R^{t,\i}(K,\mu)$ onto $H^\i(\mathcal R)$ and prove Theorem \ref{mainThm}. We discuss some applications and prove Theorem \ref{DecompTheoremRT}, Theorem \ref{thmA}, Corollary \ref{thmAB}, and Corollary \ref{thmB} in section 7.

\section{\textbf{Preliminaries and Definition of the Map $\rho$}}

For $A_1,A_2 > 0,$ the statement $A_1 \lesssim A_2$ (resp. $A_1\gtrsim A_2$) means: there exists an absolute constant $C>0$ (resp. $c>0$), independent of $A_1$ and $A_2,$ such that $A_1\le CA_2$ (resp. $A_1\ge cA_2$). We say $A_1\approx A_2$ if $A_1 \lesssim A_2$ and $A_1\gtrsim A_2.$

If $B \subset\C$ is a compact subset, then  we define the analytic capacity of $B$ by
\[
\ \gamma(B) = \sup |f'(\infty)|,
\]
where the supremum is taken over all those functions $f$ that are analytic in $\mathbb C_{\infty} \setminus B$ ($\mathbb{C}_\infty := \mathbb{C} \cup \{\infty \}$) such that
$|f(z)| \le 1$ for all $z \in \mathbb{C}_\infty \setminus B$; and
$f'(\infty) := \lim _{z \rightarrow \infty} z(f(z) - f(\infty)).$
The analytic capacity of a general subset $E$ of $\mathbb{C}$ is given by: 
$\gamma (E) = \sup \{\gamma (B) : B\subset E \text{ compact}\}.
$ The following elementary property can be found in \cite[Theorem VIII.2.3]{gamelin},
\begin{eqnarray}\label{AreaGammaEq}
\ \area(E) \le 4\pi \gamma(E)^2,
\end{eqnarray}
where $\area$ is the area (Lebesgue) measure on $\C$.
For $\nu \in M_0(\C)$ and $\epsilon > 0,$ define
\[
\ \mathcal C_\epsilon(\nu)(z) = \int _{|w-z| > \epsilon}\dfrac{1}{w - z} d\nu (w).
\] 
The (principal value) Cauchy transform
of $\nu$ is defined by
\ \begin{eqnarray}\label{CTDefinition}
\ \mathcal C(\nu)(z) = \lim_{\epsilon \rightarrow 0} \mathcal C_\epsilon(\nu)(z)
\ \end{eqnarray}
for all $z\in\mathbb{C}$ for which the limit exists.
From 
Corollary \ref{ZeroAC}, we see that \eqref{CTDefinition} is defined for all $z$ except for a set of zero analytic 
capacity.
In particular, by \eqref{AreaGammaEq}, it is defined $\area-a.a..$ Throughout this paper, the Cauchy transform of a measure always means the principal value of the transform.
In the sense of distributions,
 \begin{eqnarray}\label{CTDistributionEq}
 \ \bar \partial \mathcal C(\nu) = - \pi \nu.
 \end{eqnarray}
 Good sources for basic information about analytic
capacity and Cauchy transform are \cite{Tol14}, \cite{dud10}, \cite{gamelin}, and \cite{conway}.

The maximal Cauchy transform is defined by
 \[
 \ \mathcal C_*(\nu)(z) = \sup _{\epsilon > 0}| \mathcal C_\epsilon(\nu)(z) |.
 \]

A related capacity, $\gamma _+,$ is defined for subsets $E$ of $\mathbb{C}$ by:
\[
\ \gamma_+(E) = \sup \|\mu \|,
\]
where the supremum is taken over $\mu\in M_0^+(E)$ for which $\|\mathcal{C}(\mu) \|_{L^\infty (\mathbb{C})} \le 1.$ 
Since $\mathcal C\mu$ is analytic in $\mathbb{C}_\infty \setminus \mbox{spt}(\mu)$ and $|(\mathcal{C}(\mu)'(\infty)| = \|\mu \|$, 
we have:
$\gamma _+(E) \le \gamma (E)$.  

X. Tolsa has established the following astounding results. See \cite{Tol03} (also Theorem 6.1 and Corollary 6.3 in \cite{Tol14}) for (1) and (2). See \cite[Proposition 2.1]{Tol02} (also  \cite[Proposition 4.16]{Tol14}) for (3).

\begin{theorem}\label{TolsaTheorem}
(Tolsa 2003) There exists an absolute constant $C_T > 0$ such that:
 
(1) $\gamma_+$ and $\gamma$ are actually equivalent. 
That is,  
\[
\ \gamma (E) \le C_T \gamma_+(E)\text{ for all }E \subset \mathbb{C}.
\] 

(2) Semiadditivity of analytic capacity: for $E_1,E_2,...,E_m \subset \mathbb{C}$ ($m$ may be $\i$),
\[
\ \gamma \left (\bigcup_{i = 1}^m E_i \right ) \le C_T \sum_{i=1}^m \gamma(E_i).
\]

(3) For $a > 0$, we have:  
\[
\ \gamma(\{\mathcal{C}_*(\nu)  \geq a\}) \le C_T \dfrac{\|\nu \|}{a}.
\]
\end{theorem}

For $\eta \in M_0^+(\mathbb C)$, define 
\[
\ N_2(\eta) = \sup_{\epsilon > 0}\sup_{\|f\|_{L^2(\eta)} = 1}\|\mathcal C_\epsilon (f \eta )\|_{L^2(\eta)}.
\]
$\eta$ is of $c$-linear growth if
$\eta(\D(\lambda, \delta)) \le c \delta,\text{ for }\lambda\in \C \text{ and }\delta > 0.$ The following Proposition is from \cite[Theorem 4.14]{Tol14} and its proofs.

\begin{proposition} \label{GammaPlusThm}
There exists an absolute constant $C_T > 0$ (we use the same constant as in Theorem \ref{TolsaTheorem}) such that if $F\subset \mathbb C$ is a compact subset and $\eta\in M_0^+(F),$ then the following properties are true.
\newline
(1) If $\|\CT \eta\|_{L^\i(\C)} \le 1,$ then $\eta$ is of $1$-linear growth and $\sup_{\epsilon > 0}\|\mathcal C_\epsilon (\eta )\|_{L^\i(\C)} \le C_T.$
\newline
(2) If $\eta$ is of $1$-linear growth and $\|\mathcal C_\epsilon (\eta )\|_{L^\i(\C)} \le 1$ for all $\epsilon > 0,$  then there exists a subset $A\subset F$ such that $\eta (F) \le 2 \eta (A)$ and $N_2(\eta|_A) \le C_T.$
\newline
(3) If $N_2(\eta) \le 1,$ then there exists some function $w$ supported on $F$, with $0\le w \le 1$ such that
$\eta (F) \ \le\  2 \int w d\eta$
 and 
$\sup_{\epsilon > 0}\|\mathcal C _\epsilon (w\eta)\|_{ L^\infty (\mathbb C)} \ \le C_T.$
\end{proposition}

Combining Theorem \ref{TolsaTheorem} (1), Proposition \ref{GammaPlusThm}, and \cite[Theorem 8.1]{Tol14}, we get the following corollary. The reader may also see \cite[Corollary 3.1]{acy19}.

\begin{corollary}\label{ZeroAC}
If $\nu\in M_0(\C),$ then there exists 
$\mathcal Q \subset \mathbb{C}$ with $\gamma(\mathcal Q) = 0$ such that $\lim_{\epsilon \rightarrow 0}\mathcal{C} _{\epsilon}(\nu)(z)$ 
exists for $z\in\mathbb{C}\setminus \mathcal Q$.
\end{corollary} 

\begin{lemma}\label{ConvergeLemma} 
Let $\mu\in M_0^+(\C).$ Suppose $f_n,f\in L^1(\mu)$ satisfying: 
$\|f_n - f\|_{L^1(\mu)}\rightarrow 0$ and $ f_n \rightarrow  f, ~\mu$-a.a.
Then:

(1) For $\epsilon > 0$, there exists a subset $A_\epsilon$ with $\gamma(A_\epsilon) < \epsilon$ and a subsequence $\{f_{n_m}\}$ such that $\{\mathcal C(f_{n_m}\mu)\}$ uniformly converges to $\mathcal C(f\mu)$ on $\mathbb C \setminus A_\epsilon$.

(2) There exists a subset $\mathcal Q$ 
with $\gamma (\mathcal Q) = 0$  and a subsequence $\{f_{n_m}\}$ such that $ \mathcal C(f_{n_m}\mu)(z)$ converges to $\mathcal C(f\mu)(z)$ for $z \in \mathbb C\setminus \mathcal Q$.
\end{lemma}

\begin{proof} 

Applying Theorem \ref{TolsaTheorem} and Corollary \ref{ZeroAC}, \cite[Lemma 2.5]{y23} proves (1).

(2): Set $\mathcal Q = \cap_{k=1}^\infty A_{\frac 1k}$. Clearly, $\gamma (\mathcal Q) = 0$ and there exists a subsequence $\{f_{n_m}\}$ such that $ \mathcal C(f_{n_m}\mu)(z)$ converges to $\mathcal C(f\mu)(z)$ for $z \in \mathbb C\setminus \mathcal Q$. 
\end{proof}

From section 2 to section 6, we assume that $K\subset \C$ is a compact subset, $\mu\in M_0^+(K),$ $1 \le t <\i,$ $\frac 1t + \frac 1s = 1,$ $S_\mu$ on $R^t(K,\mu)$ is pure, and $\Lambda = \{g_j\}_{j=1}^\i \subset R^t(K,\mu)^\perp$ is a $L^s(\mu)$ norm dense subset of $R^t(K,\mu)^\perp.$ Define the non-zero set $\mathcal N$ of $\{\CT(g_j\mu)\}$ as the following:
\begin{eqnarray} \label{NZSetDef}
\ \mathcal N = \bigcup_{j=1}^\i \left\{z:~ \lim_{\epsilon\rightarrow 0}\CT_\epsilon (g_j\mu)(z) \text{ exists and }\CT (g_j\mu)(z)\ne 0 \right\}.
\end{eqnarray}
Define the zero set $\mathcal F_0$ of $\{\CT(g_j\mu)\}$ as the following:
\begin{eqnarray} \label{F0SetDef}
\ \mathcal F_0 = \bigcap_{j=1}^\i \{z:~ \lim_{\epsilon\rightarrow 0}\CT_\epsilon (g_j\mu)(z) \text{ exists and }\CT (g_j\mu)(z) = 0 \}.
\end{eqnarray}

\begin{proposition}\label{NF0Prop}
The following basic properties hold.

(1) $\mathcal N \subset K,$ $\C\setminus K \subset \mathcal F_0,$ and $ \mathcal N \cup \mathcal F_0 \approx \C,~ \gamma-a.a.$

(2) For $g\perp \rtkmu,$ we have $\mathcal C(g\mu)(z) = 0,~ \gamma |_{\mathcal F_0}-a.a.$

(3) The sets $\mathcal N$ and $\mathcal F_0$ are independent of the choices of $\Lambda$ up to a set of zero analytic capacity.	
\end{proposition}

\begin{proof}
(1) follows from Corollary \ref{ZeroAC}. There exists a subsequence $\{g_{j_k}\}$ such that $\|g_{j_k} - g\|_{L^s(\mu)} \rightarrow 0$ and $g_{j_k}(z) \rightarrow g(z),~ \mu-a.a..$
Applying Lemma \ref{ConvergeLemma} (2), we get (2). (3) is implied by (1) and (2).	
\end{proof}

\begin{lemma}\label{RhoExistLemma}
If $f\in \rtkmu,$ then there exists a unique function $\rho(f)$ defined on $\C,~\gamma-a.a.$ and there exists a subset $\mathcal Q_f\subset \mathcal N$ with $\gamma(\mathcal Q_f) = 0$ such that $\rho(f)(z) = 0$ for $z\in \mathcal F_0,$
\begin{eqnarray} \label{RhoExistLemmaEq1}
\ \rho(f)(z) \CT(g\mu)(z) = \CT(fg\mu)(z),~\gamma-a.a.\text{ for }g\perp \rtkmu,
\end{eqnarray}
and
\begin{eqnarray} \label{RhoExistLemmaEq2}
\ \rho(f)(z) = f(z),~\mu _{\mathcal N \setminus \mathcal Q_f}-a.a.	
\end{eqnarray}
Clearly, $\rho(r)(z) = r(z)$ for $z\in \mathcal N$ and $r\in \text{Rat}(K).$
\end{lemma}

\begin{proof}
Choose $\{r_n\}\subset \text{Rat}(K)$ such that 
 $\|r_n - f\|_{L^t(\mu )}\rightarrow 0\text{ and } r_n\rightarrow f,~ \mu-a.a.$
 For $\lambda \in K,$ we see that $\frac{r_n(z)-r_n(\lambda)}{z-\lambda} \in \text{Rat}(K).$ So $\int \frac{r_n(z)-r_n(\lambda)}{z-\lambda} g(z)d\mu(z) = 0$ and applying Corollary \ref{ZeroAC}, we have
 \[
 \ r_n(z) \CT(g_j\mu)(z) = \CT(r_ng_j\mu)(z),~\gamma-a.a.\text{ for }n,j \ge 1.
 \]
 From Lemma \ref{ConvergeLemma} (2), we can choose a subsequence $\{r_{n_k}\}$ such that
 \[
 \ r_{n_k}(z) \CT(g_j\mu)(z) \rightarrow \CT(fg_j\mu)(z),~\gamma-a.a.\text{ for }j \ge 1.
 \]
 Therefore, $r_{n_k}$ converges to a function, denoted $\rho(f),$ on $\mathcal N$ and
 \begin{eqnarray}\label{RhoExistLemmaEq3}	
\ \rho(f)(z) \CT(g_j\mu)(z) = \CT(fg_j\mu)(z),~\gamma-a.a.\text{ for }j \ge 1.
 \end{eqnarray}
 Set $\rho(f)(z) = 0$ for $z\in \C\setminus \mathcal N.$
 It is clear that $\rho(f)$ is unique up to a set of zero analytic capacity. Choose a subsequence $\{g_{j_k}\}$ such that $\|g_{j_k} - g\|_{L^s(\mu)}\rightarrow 0$ and $g_{j_k}(z) \rightarrow g(z),~ \mu-a.a..$ Using Lemma \ref{ConvergeLemma} (2) and \eqref{RhoExistLemmaEq3}, we get \eqref{RhoExistLemmaEq1}. 
 Let $\mathcal Q_f$ be a subset with $\gamma(\mathcal Q_f) = 0$ such that $r_{n_k}(z) \rightarrow \rho(f)(z)$ for $z\in \mathcal N \setminus \mathcal Q_f.$ \eqref{RhoExistLemmaEq2} follows since $r_{n_k}(z) \rightarrow f(z),~\mu-a.a.$
\end{proof}

\begin{proposition}\label{Rhoprop} 
The following statements are true for $\rho.$

(1) If $f_1,f_2\in R^{t,\i}(K,\mu),$ then
 $\rho(f_1f_2)(z) = \rho(f_1)(z)\rho(f_2)(z),~\gamma-a.a..$

(2) If $f\in R^{t,\i}(K,\mu)$, then $\|\rho(f)\|_{L^\infty(\area_{\mathcal N})} \le  \|f\|_{L^\infty(\mu)}$.
\end{proposition}

\begin{proof}
(1): For $f_1,f_2\in R^{t,\i}(K,\mu)$ and $g\perp R^t(K,\mu)$, we see that $f_2g\perp R^t(K,\mu)$. From \eqref{RhoExistLemmaEq1}, we get
 \[
 \ \rho  (f_1f_2)(z) \mathcal C(g\mu)(z) = \rho  (f_1)(z)\mathcal C(f_2g\mu)(z) = \rho  (f_1)(z)\rho  (f_2)(z)\mathcal C(g\mu)(z),~ \gamma - a.a.
 \]
Hence, (1) follows.

(2): For $f\in R^{t,\i}(K,\mu)$ and $g\perp R^t(K,\mu)$, using (1) and \eqref{RhoExistLemmaEq1}, we have
 \[
 \ \begin{aligned}
 \ \int_{\mathcal N} |\rho(f)(z)|^n |\mathcal C(g\mu)(z) | d \area(z) = & \int_{\mathcal N} |\mathcal C(f^ng\mu)(z) | d \area(z) \\
 \ \le & \int_{\mathcal N} \int \dfrac{|f(w)|^n |g(w)|}{|w-z|}d\mu(w) d \area(z) \\
\ \le &2\pi d\|g\|_{L^1(\mu)}\|f\|^n_{L^\infty(\mu)},
 \ \end{aligned}
 \]
where $d$ is the diameter of $\mathcal N,$ which implies 
 \[
 \ \left (\int_{\mathcal N} |\rho(f)(z)|^n |\mathcal C(g\mu)(z)| d \area(z) \right )^{\frac{1}{n}}\le \left (2\pi d\|g\|_{L^1(\mu)} \right )^{\frac{1}{n}}\|f\|_{L^\infty(\mu)}. 
\]
Taking $n\rightarrow \infty$, we get $\|\rho(f)\|_{L^\infty(\area|_{\{\mathcal C(g\mu)) \ne 0 \}})} \le  \|f\|_{L^\infty(\mu)}$. Now (2) follows from the definition \eqref{NZSetDef} of $\mathcal N.$  
\end{proof}

\section{\textbf{The Structure of $\mathcal N$}}

For $\nu\in M_0(\C),$ define the zero and non-zero linear density sets:
\[
 \ \mathcal{ZD}(\nu) = \{ \lambda:~ \Theta_\nu (\lambda ) = 0 \} \text{ and }
  \mathcal{ND}(\nu) = \{ \lambda:~ \Theta^*_\nu (\lambda ) > 0 \},
 \]
where $\Theta_\nu (\lambda ) = \lim_{\delta\rightarrow 0} \frac{|\nu |(\D(\lambda , \delta ))}{\delta}$
 if the limit exists and $\Theta_\nu^* (\lambda ) = \underset{\delta\rightarrow 0}{\overline\lim} \frac{|\nu |(\D(\lambda , \delta ))}{\delta}.$ 
 Set $\mathcal{ND}(\nu, n) =  \{ \lambda:~ n^{-1} \le \Theta^*_\nu (\lambda ) \le n \}.$
 Define
 \begin{eqnarray} \label{R0SetDef}
 \ \mathcal R_0 = \mathcal N \cap \mathcal{ZD}(\mu).	
 \end{eqnarray}
 The $1$-dimensional Hausdorff measure of $F$ is:
 \[
 \ \mathcal H^1 (F) = \sup_{\epsilon >0}\inf \left \{\sum_i \text{diam}(F_i):~ F\subset \cup_i F_i,~ \text{diam}(F_i)\le \epsilon \right \}.
 \] 
For basic information about Hausdorff measures, the reader may take a look at the books Carleson \cite{car67} and Mattila \cite{mat95}.

Let $\mathbb R$ be the real line. Let $A:\mathbb R\rightarrow \mathbb R$ be a Lipschitz function  with its graph $\Gamma = \{(x,A(x)):~ x\in \mathbb R\}.$ The following lemma, which describes the decomposition of $\mu\in M_0^+(K),$ is important for us to understand $\mathcal N.$

 \begin{lemma}\label{GammaExist}
 For $\nu \in M_0(\C),$
there is a sequence of Lipschitz functions $A_n: \mathbb R\rightarrow \mathbb R$ such that their (rotated) graphs $\Gamma_n$ with rotation angles $\beta_n$ satisfying:

(1) Let $\Gamma = \cup_n \Gamma_n$. Then
$\nu = h\mathcal H^1 |_{\Gamma} + \nu_s$ is the Radon-Nikodym decomposition with respect to $\mathcal H^1 |_{\Gamma}$, where $h\in L^1(\mathcal H^1 |_{\Gamma})$ and $\nu_s\perp \mathcal H^1 |_{\Gamma};$

(2) $\mathcal {ND}(\nu ) \approx \mathcal N (h), ~\gamma-a.a.,$ where $\mathcal N (h) := \{h\ne 0\};$

(3) $\Theta_{\nu_s}(z) = 0, ~\gamma-a.a..$ 
\newline
As a result, if $\eta\in M_0^+(\C)$ is of $c$-linear growth and $\eta\perp|\nu|$, then $\eta(\mathcal{ND}(\nu)) = 0.$
\end{lemma}

\begin{proof} Without loss of generality, we assume $\nu \in M_0^+(\C).$  
Using \cite[Lemma 8.12]{Tol14}, we see that 
\begin{eqnarray} \label{GammaExistEq1}
\ \mathcal H^1(\{\Theta^*_\nu (z) \ge n \}) \lesssim \frac{\|\nu\|}{n}
\end{eqnarray}
and $\nu |_{\mathcal{ND}(\nu, n)} = h \mathcal H^1|_{\mathcal{ND}(\nu, n)}$, where $h$ is some Borel function such that $\frac{1}{n} \lesssim h(z) \lesssim n,~ \mathcal H^1|_{\mathcal{ND}(\mu, n)}-a.a.$  and $\mathcal H^1(\mathcal{ND}(\nu, n)) < \i.$
Let $\mathcal Q_\nu = \{z:~ \Theta_\nu^*(z) = \infty\}.$ Then $\gamma(\mathcal Q_\nu) = \mathcal H^1(\mathcal Q_\nu) = 0$ by \eqref{GammaExistEq1} and 
 \[
 \ \mathcal{ND}(\nu) = \bigcup_n \mathcal{ND}(\nu, n) \cup \mathcal Q_\nu.
 \]  
	
From \cite[Theorem 1.26]{Tol14}, $\mathcal{ND}(\nu, n) = E_r\cup E_u,$ where $E_r$ is rectifiable and $E_u$ is purely unrectifiable. From David's Theorem (see \cite{dav98} or \cite[Theorem 7.2]{Tol14}), we see that $\gamma(E_u) = 0.$
Applying \cite[Proposition 4.13]{Tol14}, we find a sequence of (rotated) Lipschitz graphs $\{\Gamma_{nm}\}$ and $E_1$ with $\mathcal H^1(E_1) = 0$ such that
$E_r \subset E_1 \cup \cup_{m=1}^\infty \Gamma_{nm}.$ 
Let $\mathcal Q = E_u\cup E_1.$ Clearly, $\gamma(\mathcal Q) = 0.$ Hence, there exists a sequence of $\{\Gamma_n\}$ such that (1), (2), and (3) hold.
\end{proof}

\begin{definition}\label{GLDef}
Let  $\mathcal Q$ be a set with $\gamma(\mathcal Q) = 0$.
Let $f(z)$ be a function defined 
on $\D(\lambda, \delta_0)\setminus \mathcal Q$ for some $\delta_0 > 0.$ The function $f$ has a $\gamma$-limit $a$ at $\lambda$ if
\[  
 \  \lim_{\delta \rightarrow 0} \dfrac{\gamma(\D(\lambda, \delta) \cap \{|f(z) - a| > \epsilon\})} {\delta}= 0
\]
for all $\epsilon > 0$. If in addition, $f(\lambda)$ is well defined and $a = f(\lambda)$, then $f$ is $\gamma$-continuous at $\lambda$. 
\end{definition}
 
 The following lemma is from \cite[Lemma 3.2]{acy19}. 

\begin{lemma}\label{CauchyTLemma} 
Let $\nu\in M_0(\mathbb{C})$ and assume that for some $\lambda \in\mathbb C,$ 
$\Theta_\nu (\lambda ) = 0$ and $\mathcal{C} (\nu)(\lambda ) = \lim_{\epsilon \rightarrow 0}\mathcal{C} _{\epsilon}(\nu)(\lambda )$ exists.
Then $\mathcal{C}(\nu)(z)$ is $\gamma$-continuous at $\lambda$.
\end{lemma}

\begin{lemma} \label{LGCTL2Bounded}
 Let $A:\mathbb R\rightarrow \mathbb R$ be a Lipschitz function with its graph $\Gamma$. 
 Then 
 $N_2(\mathcal H^1 |_\Gamma) < \i,$
 while $N_2(\mathcal H^1 |_\Gamma)$ only depends on $\|A'\|_\infty$.
Consequently, there is a constant $c_\Gamma > 0$ that depends only on $\|A'\|_\infty$ such that for $E\subset \Gamma$,
 \begin{eqnarray}\label{HACEq}
 \ c_\Gamma \mathcal H^1 |_\Gamma (E) \le \gamma (E) \le \mathcal H^1 |_\Gamma (E). 
 \end{eqnarray}	
\end{lemma}

\begin{proof}  See \cite{cmm82} or \cite[Theorem 3.11]{Tol14} for $N_2(\mathcal H^1 |_\Gamma) < \i.$ Therefore, from Proposition \ref{GammaPlusThm} (3) and Theorem \ref{TolsaTheorem} (1), we conclude that $c_\Gamma \mathcal H^1 |_\Gamma (E) \le \gamma (E)$ holds. See \cite[Theorem 1.21]{Tol14} for $\gamma (E) \le \mathcal H^1 |_\Gamma (E).$
\end{proof}

Let $A$ be a Lipschitz function with its graph $\Gamma$. Define
$U_\Gamma = \{Im(z) > A(Re(z))\}$ and $L_\Gamma = \{Im(z) < A(Re(z))\}.$
	Denote $I = \sqrt{-1}.$ We consider the usual complex-valued measure
	 \[
	 \ \dfrac{1}{2\pi I} dz_{\Gamma} = \dfrac{1 + IA'(Re(z))}{2\pi I (1 + A'(Re(z))^2)^{\frac{1}{2}}} d\mathcal H^1 |_{\Gamma} = L(z)d\mathcal H^1 |_{\Gamma}.
	 \]
Notice that
$|L(z)| = \frac{1}{2\pi}.$
Plemelj's formula for Lipschitz graphs can be found as in \cite[Theorem 8.8]{Tol14}.
We now improve Plemelj's formula for an arbitrary measure as the following theorem.

\begin{theorem}\label{GPTheorem1} 
Let $A$ be a Lipschitz function with its graph $\Gamma$ and
let $\nu\in M_0(\C)$. Suppose that $\nu = b\mathcal H^1 |_{\Gamma} + \nu_s$ is the Radon-Nikodym 
decomposition with respect to $\mathcal H^1 |_{\Gamma}$, where $b\in L^1(\mathcal H^1 |_{\Gamma})$ and $\nu_s\perp \mathcal H^1 |_{\Gamma}$. Then there exists a subset $\mathcal Q\subset \mathbb C$ with $\gamma(\mathcal Q) = 0$, such that the following hold:

(a) $\mathcal C(\nu ) (\lambda) = \lim_{\epsilon\rightarrow 0} \mathcal C_{\epsilon}(\nu)(\lambda)$ exists for $\lambda\in \mathbb C\setminus \mathcal Q$.

(b) For $\lambda \in \Gamma \setminus \mathcal Q$ and $\epsilon > 0$, $v^+(\nu, \Gamma)( \lambda) := \mathcal C(\nu ) (\lambda) + \frac{b(\lambda)}{2L(\lambda)},$
\[
 \ \lim_{\delta \rightarrow 0} \dfrac{\gamma(U_\Gamma \cap \D(\lambda, \delta ) \cap \{|\mathcal{C}(\nu )(z) - v^+(\nu, \Gamma)( \lambda)| > \epsilon\})} {\delta}= 0.  
 \]
 
(c) For $\lambda \in \Gamma \setminus \mathcal Q$ and $\epsilon > 0$, $v^-(\nu, \Gamma)( \lambda) := \mathcal C(\nu ) (\lambda) - \frac{b(\lambda)}{2L(\lambda)},$
\[
 \ \lim_{\delta \rightarrow 0} \dfrac{\gamma(L_\Gamma \cap \D(\lambda, \delta ) \cap \{|\mathcal{C}(\nu )(z) - v^-(\nu, \Gamma)( \lambda)| > \epsilon\})} {\delta}= 0.  
 \]

(d) For $\lambda \in \Gamma \setminus \mathcal Q$ and $\epsilon > 0$, $v^0(\nu, \Gamma)( \lambda) := \mathcal C(\nu ) (\lambda),$
\[
 \ \lim_{\delta \rightarrow 0} \dfrac{\gamma(\Gamma \cap \D(\lambda, \delta ) \cap \{|\mathcal{C}(\nu )(z) - v^0(\nu, \Gamma)( \lambda)| > \epsilon\})} {\delta}= 0.  
 \]
\end{theorem}

\begin{proof}
As $\nu$ is compactly supported, 
 we will just consider the portion  of the graph of $\Gamma$ that lies in $\{z:~ |Re(z)| < M\}$ for some $M > 0.$ 

From Corollary \ref{ZeroAC}, (a) follows and we assume that $\mathcal C(\nu_s ) (\lambda) = \lim_{\epsilon\rightarrow 0} \mathcal C_{\epsilon}(\nu_s)(\lambda),$  $\mathcal C(b\mathcal H^1 |_{\Gamma}) (\lambda) = \lim_{\epsilon\rightarrow 0} \mathcal C_{\epsilon}(b\mathcal H^1 |_{\Gamma})(\lambda),$ and $\mathcal C(dz_{\Gamma}) (\lambda) = \lim_{\epsilon\rightarrow 0} \mathcal C_{\epsilon}(dz_{\Gamma})(\lambda)$ exist.
 Using Lemma \ref{GammaExist}, we assume that $\Theta_{\nu_s} (\lambda) = 0.$
We also assume that $\lambda$ is a differentiable point for $\Gamma$ and is a Lebesgue point for $\frac{b}{L}$, that is,
\begin{eqnarray}\label{GPTheoremEq1}
 \ \lim_{\delta\rightarrow 0} \dfrac{1}{\delta}\int_{\D(\lambda,\delta)} \left |\frac{b(z)}{L(z)} - \frac{b(\lambda)}{L(\lambda)} \right ||dz_{\Gamma}| = 0 
\end{eqnarray}  
as $\int_{\D(\lambda,\delta)}|dz_{\Gamma}| \approx \delta.$ Set $\nu = \nu_1 + \frac{b(\lambda)}{2\pi IL(\lambda)}dz_{\Gamma},$ where $\nu_1 = (\frac{b}{2\pi IL} - \frac{b(\lambda)}{2\pi IL(\lambda)}) dz_{\Gamma} + \nu_s.$ From \eqref{GPTheoremEq1}, we see that $\Theta_{\nu_1} (\lambda) = 0.$ Applying Lemma \ref{CauchyTLemma} to $\nu_1,$ we conclude that $\mathcal{C}(\nu_1)$ is $\gamma$-continuous at $\lambda.$ Therefore, using Theorem \ref{TolsaTheorem} (2), we just need to prove (b), (c), and (d) for $\nu = dz_{\Gamma}$ below.

(b): Let $\lambda_l$ and $\lambda_r$ be the left and right intersections of $\Gamma$ and $\partial \D(\lambda,\delta_0),$ respectively. Let $B_u$ denote the boundary of $U_\Gamma \cap \D(\lambda, \delta_0 ).$ Then for $\delta < \delta_0$ and $w \in U_\Gamma \cap \D(\lambda, \delta),$
\[
\ \left | \int _{B_u \setminus \Gamma} \dfrac{dz}{z - \lambda} - \int _{B_u \setminus \Gamma} \dfrac{dz}{z - w} \right | \lesssim \dfrac{\delta}{\delta_0 - \delta}.
\]
Hence,
\[
\begin{aligned}
\ \CT(dz_\Gamma)(w) - \int _{|z-\lambda| > \delta_0} \dfrac{dz_\Gamma}{z - w}  & =  \int _{B_u} \dfrac{dz}{z - w} - \int _{B_u \setminus \Gamma} \dfrac{dz}{z - w} \\
\ &\rightarrow 2\pi I - \int _{B_u \setminus \Gamma} \dfrac{dz}{z - \lambda} \\
\ & = 2\pi I - I(arg(\lambda_l - \lambda) - arg(\lambda_r - \lambda))
\end{aligned}
\]
as $\delta\rightarrow 0$ ($w\rightarrow \lambda$).
Since $arg(\lambda_l - \lambda) - arg(\lambda_r - \lambda) \rightarrow \pi$ and $\CT_{\delta_0} (dz_\Gamma)(\lambda) \rightarrow \CT (dz_\Gamma)(\lambda)$ as $\delta_0\rightarrow 0,$ we have 
\[
\ \lim_{\delta\rightarrow 0,~ w \in U_\Gamma \cap \D(\lambda, \delta)} \CT  (dz_\Gamma)(w) = \CT (dz_\Gamma)(\lambda) + \pi I.
\]
This completes the proof of (b). The proof of (c) is the same.

From Lemma \ref{LGCTL2Bounded}, we see  that 
the Cauchy transform of $\mathcal H^1 |_{\Gamma}$ is bounded on $L^2(\mathcal H^1 |_{\Gamma})$. So we get $\mathcal C(dz_{\Gamma})\in L^2(\mathcal H^1 |_{\Gamma})$ as $\Gamma \subset \{z:~ |Re(z)| < M\}$. If $\lambda$ is a Lebesgue point for $\mathcal C(dz_{\Gamma}),$ then we have
 \[
\ \begin{aligned}
 \ & \dfrac{\mathcal H^1 |_{\Gamma} \left ( \left \{|\mathcal C(dz_{\Gamma})  - \mathcal C(dz_{\Gamma}) (\lambda) | > \epsilon \right \}\cap \D(\lambda, \delta)\right)}{\delta} \\
 \ \lesssim &\dfrac{1}{\epsilon\delta}\int _{\D(\lambda, \delta)\cap \Gamma} |\mathcal C(dz_{\Gamma})(z)  - \mathcal C(dz_{\Gamma}) (\lambda) | d\mathcal H^1 \rightarrow 0
 \end{aligned}
 \]
 as $\delta\rightarrow 0,$ which proves (d) by \eqref{HACEq}.
\end{proof}

We point out if $\Gamma$ is a rotated Lipschitz graph (with rotation angle $\beta$ and Lipschitz function $A$) at $\lambda_0$, then
 \begin{eqnarray}\label{VPlusBeta}
 \ v^+(\nu, \Gamma, \beta)( \lambda) := \mathcal C(\nu ) (\lambda) + \frac{e^{-i\beta}b(\lambda)}{2L((\lambda - \lambda_0)e^{-i\beta} + \lambda_0)}.
 \end{eqnarray}
Similarly,
 \begin{eqnarray}\label{VMinusBeta}
 \ v^-(\nu, \Gamma, \beta)( \lambda) := \mathcal C(\nu ) (\lambda) - \frac{e^{-i\beta}b(\lambda)}{2L((\lambda - \lambda_0)e^{-i\beta} + \lambda_0)}.
 \end{eqnarray}
 Hence, $v^+(\nu, \Gamma)( \lambda) = v^+(\nu, \Gamma, 0)( \lambda)$ and $v^-(\nu, \Gamma)( \lambda) = v^-(\nu, \Gamma, 0)( \lambda).$ 
 
 From now on, we fix $\Gamma_n$ with rotation angle $\beta_n,$ $\Gamma,$ and $h$ as in Lemma \ref{GammaExist} for $\mu\in M_0^+(K)$. We may assume that $\mathcal H^1(\Gamma_n \cap \Gamma_m) = 0$ for $n \ne m.$ Otherwise, we just consider the portion $\Gamma_n \setminus \cup_{k=1}^{n-1}\Gamma_k.$ Define
 \[
 \ \mathcal Z_+(\Gamma_n) = \bigcap_{j = 1}^\infty \left \{z\in \Gamma_n \cap \mathcal {ND}(\mu): ~ v^+(g_j\mu, \Gamma_n, \beta_n)(z) = 0 \right \}, 
 \]
 \[
 \ \mathcal Z_-(\Gamma_n) = \bigcap_{j = 1}^\infty \left \{z\in \Gamma_n  \cap \mathcal {ND}(\mu): ~ v^-(g_j\mu, \Gamma_n, \beta_n)(z) = 0 \right \},
 \]
 \[
 \ \mathcal N_+(\Gamma_n) = \bigcup_{j = 1}^\infty \left \{z\in \Gamma_n  \cap  \mathcal {ND}(\mu): ~ v^+(g_j\mu, \Gamma_n, \beta_n)(z) \ne 0 \right \},
 \]
 and 
 \[
 \ \mathcal N_-(\Gamma_n) = \bigcup_{j = 1}^\infty \left \{z\in \Gamma_n  \cap  \mathcal {ND}(\mu): ~ v^-(g_j\mu, \Gamma_n, \beta_n)(z) \ne 0 \right \}.
 \]
 
 \begin{proposition}\label{NZGammaProp}
 For $n\ge 1,$ the following statements are true.
 
 (1) If $g\perp \rtkmu,$ then 
 \[
 \begin{aligned}
 \ &v^+(g\mu, \Gamma_n, \beta_n)(z) = 0,~\mathcal H^1 |_{\mathcal Z_+(\Gamma_n)}-a.a., \\
 \ &v^-(g\mu, \Gamma_n, \beta_n)(z) = 0,~\mathcal H^1 |_{\mathcal Z_-(\Gamma_n)}-a.a.
 \end{aligned}
 \]
 Consequently, $\mathcal Z_+(\Gamma_n)$ and $\mathcal Z_-(\Gamma_n)$	 are independent of the choices of $\{g_j\}$ up to a set of zero analytic capacity.
 
 (2) $\mathcal N_+(\Gamma_n) \cap \mathcal N_-(\Gamma_n) \approx (\Gamma_n  \cap  \mathcal {ND}(\mu)) \setminus (\mathcal Z_+(\Gamma_n) \cup \mathcal Z_-(\Gamma_n)),~\mathcal H^1|_{\Gamma_n}-a.a..$ Consequently, $\mathcal N_+(\Gamma_n) \cap \mathcal N_-(\Gamma_n)$	 is independent of the choices of $\{g_j\}$ up to a set of zero analytic capacity.
 \end{proposition}
 
 \begin{proof}
 There exists a subsequence $\{g_{j_k}\}$ such that $\|g_{j_k} - g\|_{L^s(\mu)} \rightarrow 0$ and $g_{j_k}(z) \rightarrow g(z),~ \mu-a.a..$
Applying Lemma \ref{ConvergeLemma} (2), we see (1) and (2) follow.	
 \end{proof}

We now define: 
\begin{eqnarray}\label{FPMSetDef}
 \ \mathcal F_+  =  \bigcup_{n = 1}^\infty \mathcal Z_+(\Gamma_n) \text{ and }\mathcal F_- =  \bigcup_{n = 1}^\infty \mathcal Z_-(\Gamma_n)
 \end{eqnarray}
 and
 \begin{eqnarray}\label{R1SetDef}
 \ \mathcal R_1 =  \bigcup_{n = 1}^\infty \left ( \mathcal N_+(\Gamma_n)\cap \mathcal N_-(\Gamma_n) \right ).
 \end{eqnarray}
 Recall $\mathcal R_0$ and $\mathcal F_0$ are defined as in \eqref{R0SetDef} and \eqref{F0SetDef}, respectively. We define:
 \begin{eqnarray}\label{FRSetDef}
 \ \mathcal F = \mathcal F_0\cup \mathcal F_+ \cup \mathcal F_- \text{ and } \mathcal R =  \mathcal R_0 \cup \mathcal R_1.
 \end{eqnarray}
We call $\ \mathcal F$ and $\ \mathcal R$ the non-removable boundary and removable set for $\rtkmu,$ respectively. As a simple application of Proposition \ref{NF0Prop} (2) and the fact that $\area (\mathcal F_+\cup \mathcal F_-) = 0$, we have the following property.
\begin{eqnarray} \label{acZero}
\ \mathcal C(g\mu)(z) = 0, ~ \area |_{\mathcal F}-a.a. \text{ for }g\perp R^t(K, \mu).
\end{eqnarray}

The corollary below follows from Propositions \ref{NF0Prop} \& \ref{NZGammaProp}.

\begin{corollary}\label{FRUniqueCor}
The sets $\mathcal F_0,~ \mathcal F_+,~ \mathcal F_-,~ \mathcal F,~ \mathcal R_0,~ \mathcal R_1,$ and $\mathcal R$ are independent of the choices of $\{g_j\}$ up to a set of zero analytic capacity.	
\end{corollary}

In the remaining section, we prove the following decomposition for $\mathcal N$ and
discuss a characterization of $\mathcal F$ and $\mathcal R$ which implies $\mathcal F$ and $\mathcal R$ are independent of the choices of $\{\Gamma_n\}$ up to a set of zero analytic capacity. 

\begin{theorem}\label{NDecompThm}
Let $\mathcal R_0,$ $\mathcal R,$ $\mathcal F_0,$ $\mathcal F_+,$ and $\mathcal F_-$ be deined as above. Then
\[
\ \mathcal N \approx \mathcal R \cup \mathcal F_+ \cup \mathcal F_- \text{ and } \mathcal F_0 \cup \mathcal R_0 \approx \mathcal {ZD}(\mu),~ \gamma-a.a.
\]	
\end{theorem}
Before proving Theorem \ref{NDecompThm}, we need a couple of lemmas. For $\nu\in M_0(\C),$ the maximal function of $\nu$ is defined by
 \begin{eqnarray}\label{MaxFunctDef}
 \ \mathcal M_\nu(z) = \sup _{\epsilon > 0}\dfrac{|\nu|(\D(z,\epsilon))}{\epsilon}.
 \end{eqnarray}
 
 Combining \cite[Theorem 2.5]{Tol14}, Proposition \ref{GammaPlusThm}, and Theorem \ref{TolsaTheorem} (1), we see that there exists an absolute constant $C_T$ (we use the same constant as in Theorem \ref{TolsaTheorem}) such that for $a > 0,$
\begin{eqnarray} \label{MaxFunctEq}
\ \gamma\{\lambda:~ \mathcal M_\nu (\lambda) > a\} \le \frac{C_T}{a}\|\nu\|.\ 
\end{eqnarray}

\begin{lemma} \label{CTMaxFunctFinite}
Let $\{\nu_j\} \subset M_0(\C).$ Then for $\epsilon > 0$, there exists a Borel subset $F$ such that $\gamma (F^c) < \epsilon$ and $\mathcal C_*(\nu_j)(z),~\mathcal M_{\nu_j} (z) \le M_j < \infty$ for $z \in F$.
\end{lemma}

\begin{proof}
Let $A_j = \{\mathcal C_*(\nu_j)(z) \le M_j\}$ and $B_j = \{\mathcal M_{\nu_j} (z) \le M_j\}.$ By Theorem \ref{TolsaTheorem} (3) and \eqref{MaxFunctEq}, we can select $M_j>0$ so that $\gamma(A_j^c) < \frac{\epsilon}{2^{j+2}C_T}$ and $\gamma(B_j^c) < \frac{\epsilon}{2^{j+2}C_T}.$ Set $F = \cap_{j=1}^\infty (A_j\cap B_j)$. Then applying Theorem \ref{TolsaTheorem} (2), we get
\[
 \ \gamma (F^c) \le C_T \sum_{j=1}^\infty (\gamma(A_j^c) + \gamma(B_j^c)) < \epsilon.
 \]
\end{proof}

The following lemma is a simple application of Theorem \ref{TolsaTheorem} and Proposition \ref{GammaPlusThm} (also see \cite[Corollary 2.4]{y23}).

\begin{lemma}\label{ZeroACEta}
Let $\eta\in M_0^+(\C)$ such that $\|\mathcal C (\eta)\| \le 1$. If $F$ is a compact subset and $\gamma(F) = 0$, then $\eta(F) = 0$.
\end{lemma}

 \begin{lemma} \label{BBFunctLemma}
Suppose that $\{u_n\}\subset L^1(\mu)$ and $E$ is a compact subset with $\gamma(E) > 0$. Then there exists $\eta \in M_0^+(E)$ satisfying:

(1) $\eta$ is of $1$-linear growth, $\|\mathcal C_{\epsilon}(\eta)\|_{L^\infty (\mathbb C)}\le 1$ for all $\epsilon > 0,$ and 
$\gamma(E) \lesssim \|\eta\|$;

(2) $\mathcal C_*(u_n\mu ),~ \mathcal M_{u_n\mu} \in L^\infty(\eta)$;

(3) there exists a subsequence $f_k(z) = \mathcal C_{\epsilon_k}(\eta)(z)$ such that $f_k$ converges to $f\in L^\infty(\mu)$ in weak-star topology and $f_k(\lambda) $ converges to $f(\lambda) = \mathcal C(\eta)(\lambda)$ uniformly on any compact subset of $\C\setminus \text{spt}\eta$ as $\epsilon_k\rightarrow 0$.
Moreover for $n \ge 1,$ 
\begin{eqnarray}\label{BBFunctLemmaEq1}
 \ \int f(z) u_n(z)d\mu (z) = - \int \mathcal C(u_n\mu) (z) d\eta (z),
 \end{eqnarray}
and for $\lambda\in \C\setminus \text{spt}\eta$,
 \begin{eqnarray}\label{BBFunctLemmaEq2}
 \ \int \dfrac{f(z) - f(\lambda)}{z - \lambda} u_n(z)d\mu (z) = - \int \mathcal C(u_n\mu) (z) \dfrac{d\eta (z)}{z - \lambda}.
 \end{eqnarray} 
\end{lemma}

\begin{proof}
From Lemma \ref{CTMaxFunctFinite}, we find $E_1\subset E$ such that $\gamma(E\setminus E_1) < \frac{\gamma(E)}{2C_T}$ and $\mathcal C_*(u_n\mu )(z), ~\mathcal M_{u_n\mu}(z) \le M_n < \infty$ for $z\in E_1$. Using Theorem \ref{TolsaTheorem} (2), we get
 $\gamma(E_1) \ge \frac{1}{C_T}\gamma(E) - \gamma(E\setminus E_1) \ge \frac{1}{2C_T}\gamma(E).$
Using Theorem \ref{TolsaTheorem} (1) and Proposition \ref{GammaPlusThm} (1), we infer 
that there is $\eta\in M_0^+(E_1)$ satisfying (1). So (2) holds. Clearly,
 \begin{eqnarray}\label{lemmaBasicEq3}
 \  \int \mathcal C_\epsilon(\eta)(z) u_nd\mu  = - \int \mathcal C_\epsilon(u_n\mu)(z) d\eta
 \end{eqnarray}
for $n \ge 1$. We can choose a sequence $f_k(\lambda) = \mathcal C_{\epsilon_k}(\eta)(\lambda)$ that converges to $f$ in $L^\infty(\mu)$ weak-star topology and $f_k(\lambda)$ uniformly tends to $f(\lambda)$ on any compact subset of $\C\setminus \text{spt}\eta$. On the other hand, $|\mathcal C_{\epsilon_k}(u_n\mu)(z) | \le M _n,~ \eta-a.a.$ and by Corollary \ref{ZeroAC} and Lemma \ref{ZeroACEta}, $\lim_{k\rightarrow \infty} \mathcal C_{\epsilon_k}(u_n\mu)(z)  = \mathcal C(u_n\mu)(z) ,~ \eta -a.a.$. We apply the Lebesgue dominated convergence theorem to the right hand side of \eqref{lemmaBasicEq3} and get \eqref{BBFunctLemmaEq1} for $n \ge 1$. For \eqref{BBFunctLemmaEq2}, let $\lambda\notin \text{spt}\eta$ and $d = \text{dist}(\lambda, \text{spt}\eta)$,
for $z\in \D(\lambda, \frac{d}{2})$ and $\epsilon < \frac{d}{2}$, we have
 \[
 \ \left |\dfrac{\mathcal C_\epsilon (\eta)(z) - f(\lambda)}{z - \lambda} \right |\le \left |\mathcal C_\epsilon  \left (\dfrac{\eta (s)}{s - \lambda} \right ) (z) \right | \le 
\dfrac{2}{d^2}\|\eta\|. 
 \]
For $z\notin \D(\lambda, \frac{d}{2})$ and $\epsilon < \frac{d}{2}$,
 \[
 \ \left |\dfrac{\mathcal C_\epsilon (\eta)(z) - f(\lambda)}{z - \lambda} \right | \le \dfrac{4}{d}.
 \]
Thus, we replace the above proof for the measure $\frac{\eta (s)}{s - \lambda}$. In fact, we choose a subsequence  $\{\mathcal C_{\epsilon_{k_j}} (\eta)\}$ such that $e_{k_j}(z) = \frac{\mathcal C_{\epsilon_{k_j}} (\eta)(z) - f(\lambda)}{z - \lambda}$ converges to $e(z)$ in weak-star topology. Clearly, $(z-\lambda)e_{k_j}(z)  + f(\lambda) = \mathcal C_{\epsilon_{k_j}} (\eta)(z)$ converges to $(z-\lambda)e(z)  + f(\lambda) = f(z)$ in weak-star topology.  
On the other hand, \eqref{lemmaBasicEq3} becomes
 \begin{eqnarray}\label{lemmaBasicEq31}
 \  \int \mathcal C_{\epsilon_{k_j}}(\dfrac{\eta(s)}{s-\lambda})(z) u_nd\mu  = - \int \mathcal C_{\epsilon_{k_j}}(u_n\mu)(z) \dfrac{d\eta(z)}{z-\lambda}
 \end{eqnarray}
and for $\epsilon_{k_j} < \frac{d}{2}$, we have
 \begin{eqnarray}\label{lemmaBasicEq32}
 \left | \mathcal C_{\epsilon_{k_j}}(\dfrac{\eta(s)}{s-\lambda})(z) - e_{k_j}(z) \right |
\ \le \begin{cases}0, & z\in \D(\lambda, \frac{d}{2}), \\ \dfrac{2}{d^2} \eta(\D(z, \epsilon_{k_j})), & z\notin \D(\lambda, \frac{d}{2}), \end{cases}
\end{eqnarray}
which goes to zero as $\epsilon_{k_j} \rightarrow 0$. Combining \eqref{lemmaBasicEq31}, \eqref{lemmaBasicEq32}, and the Lebesgue dominated convergence theorem, we prove the equation \eqref{BBFunctLemmaEq2}. (3) is proved.    
\end{proof}

For $n \ge 1,$ we define
 \[
 \ \mathcal N_0(\Gamma_n) = \bigcup_{j = 1}^\infty \{z\in \Gamma_n \cap \mathcal {ND}(\mu): ~ \CT (g_j\mu)(z) \ne 0 \}.
 \]
 
 \begin{lemma} \label{NPMLemma}
 For $n \ge 1,$ we have
 \[
 \ \mathcal N_+(\Gamma_n) \cup \mathcal N_-(\Gamma_n) \approx \mathcal N_0(\Gamma_n),~\mathcal H^1|_{\Gamma_n}-a.a..
 \]	
 \end{lemma}
 
 \begin{proof}
 Without loss of generality, we assume $n=1$ and $\beta_1=0.$
 Suppose there exists a compact subset $E \subset \mathcal N_+(\Gamma_1) \cup \mathcal N_-(\Gamma_1) \setminus \mathcal N_0(\Gamma_1)$ such that $\mathcal H^1(E) > 0.$ Then $\CT (g_j\mu)(z) = 0,~\mathcal H^1|_E-a.a.$ and 
 \begin{eqnarray}\label{NPMLemmaEq1}
 \ v^+(g_j\mu,\Gamma_1)(z) = \dfrac{(g_jh)(z)}{2L(z)},~ v^-(g_j\mu,\Gamma_1)(z) = - \dfrac{(g_jh)(z)}{2L(z)},~\mathcal H^1|_E-a.a.
 \end{eqnarray}
 for all $j\ge 1.$ Let $\eta	$ and $f$ be from Lemma \ref{BBFunctLemma} for $\{g_j\}$ and $E$ as $\gamma(E) > 0$ by \eqref{HACEq}. We may assume $\eta = w\mathcal H^1|_E,$ where $0\le w(z) \le 1$ on $E,$ since $\eta$ is of $1$-linear growth. From Lemma \ref{BBFunctLemma} (3), we see $f\in R^{t,\i}(K,\mu)$ as $\{g_j\}$ is dense in $R^t(K,\mu)^\perp$ and
 \begin{eqnarray}\label{NPMLemmaEq2}
 \ \CT\eta (\lambda)\CT (g_j\mu)(\lambda) = \CT (fg_j\mu)(\lambda), ~\gamma|_{E^c}-a.a.,~\text{ for } j\ge 1. 
 \end{eqnarray}
 From Proposition \ref{NF0Prop} (2), since $fg_j \perp \rtkmu,$ we get $\mathcal C(fg_j\mu)(z) = 0,~\eta-a.a.,$
 \begin{eqnarray}\label{NPMLemmaEq3}
 \ v^+(fg_j\mu,\Gamma_1)(z) = \dfrac{(fg_jh)(z)}{2L(z)},~ v^-(fg_j\mu,\Gamma_1)(z) = - \dfrac{(fg_jh)(z)}{2L(z)},~\mathcal \eta-a.a.
 \end{eqnarray}
 for $j \ge 1.$ Applying Theorem \ref{GPTheorem1} to $\CT\eta (\lambda)$ for \eqref{NPMLemmaEq2}, we get
 \[
 \begin{aligned}
 \ v^+(w\mathcal H^1,\Gamma_1)(z)v^+(g_j\mu,\Gamma_1)(z) = & v^+(fg_j\mu,\Gamma_1)(z),~ \eta-a.a., \\
 \ v^-(w\mathcal H^1,\Gamma_1)(z)v^-(g_j\mu,\Gamma_1)(z) = & v^-(fg_j\mu,\Gamma_1)(z),~ \eta-a.a.
 \end{aligned}
\]
Combining with \eqref{NPMLemmaEq1} and \eqref{NPMLemmaEq3}, we have
$g_j(z)h(z) = 0, ~\eta-a.a. \text{ for } j \ge 1,$
which implies
$v^+(g_j\mu,\Gamma_1)(z) = v^-(g_j\mu,\Gamma_1)(z) = 0, ~\eta-a.a.$ for $j \ge 1.$
This is a contradiction. 

On the other hand, if $\CT(g_j\mu)(\lambda) \ne 0,$ then $v^+(g_j\mu,\Gamma_1)(\lambda) \ne 0$ or $ v^-(g_j\mu,\Gamma_1)(\lambda) \ne 0.$ Therefore, $\mathcal N_0(\Gamma_n) \subset \mathcal N_+(\Gamma_n) \cup \mathcal N_-(\Gamma_n),~\mathcal H^1|_{\Gamma_n}-a.a..$
The lemma is proved.
 \end{proof}
 
 \begin{proof} (Theorem \ref{NDecompThm})
 If $\lambda \in \mathcal Z_+(\Gamma_n) \cap \mathcal Z_-(\Gamma_n),$ then $g_j(\lambda)h(\lambda) = 0$ for $j\ge 1.$ Hence, $\mathcal H^1(\mathcal Z_+(\Gamma_n) \cap \mathcal Z_-(\Gamma_n)) = \emptyset$ since $S_\mu$ is pure.
 From Lemma \ref{NPMLemma}, we have
 \[
 \ \begin{aligned}
 \ \mathcal N_0(\Gamma_n)& \approx \mathcal N_+(\Gamma_n) \cup \mathcal N_-(\Gamma_n) \approx \Gamma_n \cap \mathcal {ND}(\mu) \\
 \ &\approx (\mathcal N_+(\Gamma_n) \cap \mathcal N_-(\Gamma_n)) \cup \mathcal Z_+(\Gamma_n) \cup \mathcal Z_-(\Gamma_n),~ \mathcal H^1|_{\Gamma_n}-a.a. 
 \ \end{aligned}
 \]	
 Therefore, $\mathcal {ND}(\mu) \subset \mathcal N.$ The theorem follows since
 $\mathcal N = \mathcal R_0 \cup\cup _{n=1}^\i\mathcal N_0(\Gamma_n).$
 \end{proof}

Define
 \[
\ \mathcal E_N = \left \{\lambda : ~\lim_{\epsilon \rightarrow 0} \mathcal C_\epsilon(g_j\mu)(\lambda )\text{ exists, } \max_{1\le j\le N} |\mathcal C (g_j\mu)(\lambda ) | \le \frac{1}{N} \right \}.
 \]

\begin{theorem}\label{FCharacterization}
There is a subset $\mathcal Q \subset \mathbb C$ with $\gamma(\mathcal Q) = 0$ such that if $\lambda \in\mathbb C \setminus \mathcal Q$, then $\lambda \in \mathcal F$ if and only if   
 \begin{eqnarray}\label{FCEq4}
 \ \underset{\delta\rightarrow 0}{\overline{\lim}}\dfrac{\gamma(\D(\lambda, \delta)\cap\mathcal E_N)}{\delta} > 0 
 \end{eqnarray}
for all $N \ge 1$. Consequently, $\mathcal F$ and $\mathcal R$ do not depend on the choices of $\{\Gamma_n\}$ up to a set of zero analytic capacity. 
\end{theorem}

\begin{proof}
We first prove that there exists $\mathcal Q_1$ with $\gamma(\mathcal Q_1) = 0$ such that if $\lambda\in \mathcal {ZD}(\mu) \setminus \mathcal Q_1$, then $\lambda\in \mathcal F$ if and only if $\lambda$ satisfies \eqref{FCEq4}.

From Theorem \ref{NDecompThm}, $\mathcal {ZD}(\mu) \approx \mathcal R_0 \cup \mathcal F_0, ~\gamma-a.a.$. There exists $\mathcal Q_1$ with $\gamma(\mathcal Q_1) = 0$ such that for $\lambda\in \mathcal {ZD}(\mu) \setminus \mathcal Q_1$, $\Theta_{ g_j\mu}(\lambda) = 0$,
$\mathcal C(g_j\mu) (\lambda) = \lim_{\epsilon\rightarrow 0} \mathcal C_\epsilon (g_j\mu) (\lambda)$ exists, and $\mathcal C(g_j\mu) (z)$ is $\gamma$-continuous at $\lambda$ (Lemma \ref{CauchyTLemma})  for all $j\ge 1$.

If $\lambda\in \mathcal R_0$, then there exists $j_0$ such that $\mathcal C(g_{j_0}\mu) (\lambda) \ne 0$. Let $\epsilon_0 = \frac 12 |\mathcal C(g_{j_0}\mu) (\lambda)|$, we obtain that, for $N > N_0 := \max(j_0,\frac{1}{\epsilon_0} + 1)$,
 \[
 \ \mathcal E_N \subset \{|\mathcal C(g_{j_0}\mu) (z) - \mathcal C(g_{j_0}\mu) (\lambda)| > \epsilon_0 \}.
 \]
Therefore, by Lemma \ref{CauchyTLemma}, for $N > N_0$,  
 \[
 \ \lim_{\delta\rightarrow 0} \dfrac{\gamma (\D(\lambda, \delta ) \cap \mathcal E_N)}{\delta} = 0.
 \]
Thus, $\lambda$ does not satisfy \eqref{FCEq4}. 

Now for $\lambda\in \mathcal F_0$, $\mathcal C(g_j\mu) (\lambda) = 0$ for all $j \ge 1$.
Using Lemma \ref{CauchyTLemma} and  Theorem \ref{TolsaTheorem} (2), we get
 \[
 \ \lim_{\delta\rightarrow 0} \dfrac{\gamma (\D(\lambda, \delta) \setminus \mathcal E_N)}{\delta } \le C_T \sum_{j=1}^N \lim_{\delta\rightarrow 0} \dfrac{\gamma (\D(\lambda, \delta) \cap \{|\mathcal C(g_j\mu) - \mathcal C(g_j\mu) (\lambda)| \ge \frac{1}{N}\})}{\delta } = 0. 
 \]
Hence, $\lambda$ satisfies \eqref{FCEq4}.

We now prove that there exists $\mathcal Q_2$ with $\gamma(\mathcal Q_2) = 0$ such that if $\lambda\in \mathcal {ND}(\mu) \setminus \mathcal Q_2$, then $\lambda\in \mathcal F$ if and only if $\lambda$ satisfies \eqref{FCEq4}.

From \eqref{FPMSetDef} and \eqref{R1SetDef}, we get $\mathcal {ND}(\mu) \approx \mathcal R_1  \cup (\mathcal F_+ \cup \mathcal F_-), ~\gamma-a.a.$. There exists $\mathcal Q_2$ with $\gamma(\mathcal Q_2) = 0$ such that for $\lambda\in \mathcal {ND}(\mu) \cap \Gamma_n\setminus \mathcal Q_2$, 
$v^0(g_j\mu, \Gamma_n)(\lambda) = \mathcal C(g_j\mu) (\lambda) = \lim_{\epsilon\rightarrow 0} \mathcal C_\epsilon (g_j\mu) (\lambda)$, $v^+(g_j\mu, \Gamma_n, \beta_n)(\lambda)$, and $v^-(g_j\mu, \Gamma_n, \beta_n)(\lambda)$ exist for all $j,n \ge 1$ and Theorem \ref{GPTheorem1} (b), (c), and (d) hold. Fix $n = 1$ and without loss of generality, we assume $\beta_1 =0.$

If $\lambda\in \mathcal R_1$, then there exist integers $j_0$, $j_1$, and $j_2$ such that $v^0(g_{j_0}\mu, \Gamma_1)(\lambda) \ne 0$ by Lemma \ref{NPMLemma}, $v^+(g_{j_1},\Gamma_1)(\lambda) \ne 0$, and $v^-(g_{j_2}\mu, \Gamma_1)(\lambda) \ne 0$. Set
 \[
 \ \epsilon_0 = \dfrac 12 \min (|v^0(g_{j_0}\mu, \Gamma_1)(\lambda)|, |v^+(g_{j_1},\Gamma_1)(\lambda)|, |v^-(g_{j_2}\mu, \Gamma_1)(\lambda)|), 
 \]
then for $N > N_0 := \max (j_0, j_1, j_2, \frac{1}{\epsilon_0} + 1)$,
\[
\ \Gamma_1\cap \mathcal E_N \subset D := \Gamma_1 \cap \{|\mathcal{C}(g_{j_0}\mu) (z ) - v^0(g_{j_0}\mu, \Gamma_1)(\lambda)| \ge \epsilon_0 \},
 \] 
\[
\ U_{\Gamma_1}\cap \mathcal E_N \subset
 E := U_{\Gamma_1}\cap \{|\mathcal{C}(g_{j_1}\mu)(z ) - v^+(g_{j_1}\mu, \Gamma_1)(\lambda)| \ge \epsilon_0 \},\text{ and }
 \]
 \[
\ L_{\Gamma_1}\cap \mathcal E_N \subset
 F :=  L_{\Gamma_1}\cap \{|\mathcal{C}(g_{j_2}\mu)(z ) - v^-(g_{j_2}\mu, \Gamma_1)(\lambda)| \ge \epsilon_0 \}.
 \]
Therefore, using Theorem \ref{TolsaTheorem} (2) and Theorem \ref{GPTheorem1}, we get for $N > N_0$,
 \[
 \begin{aligned}
 \ &\lim_{\delta\rightarrow 0}\dfrac{\gamma(\D(\lambda, \delta) \cap \mathcal E_N )}{\delta} \\
 \ \le & C_T\left ( \lim_{\delta\rightarrow 0}\dfrac{\gamma(\D(\lambda, \delta) \cap D)}{\delta} + \lim_{\delta\rightarrow 0}\dfrac{\gamma(\D(\lambda, \delta) \cap E)}{\delta} + \lim_{\delta\rightarrow 0}\dfrac{\gamma(\D(\lambda, \delta) \cap F)}{\delta}\right ) \\
 \ = &0.
 \end{aligned}
 \]
Hence, $\lambda$ does not satisfy \eqref{FCEq4}.   
 
For $\lambda\in (\mathcal F_+ \cup \mathcal F_-) \cap \Gamma_1$, we may assume that $\lambda\in \mathcal Z_+(\Gamma_1)\subset \Gamma_1$. Using Theorem \ref{TolsaTheorem} (2) and Theorem \ref{GPTheorem1}, we get ($v^+(g_j\mu, \Gamma_1)(\lambda) = 0$)
 \[
 \ \begin{aligned}
 \ &\lim_{\delta\rightarrow 0}\dfrac{\gamma(\D(\lambda, \delta) \cap U_{\Gamma_1}\setminus \mathcal E_N)}{\delta} \\
 \ \le & C_T \sum_{j=1}^N\lim_{\delta\rightarrow 0}\dfrac{\gamma(\D(\lambda, \delta) \cap U_{\Gamma_1}\cap \{|\mathcal{C}(g_j\mu)(z ) - v^+(g_j\mu, \Gamma_1)( \lambda)| \ge \frac{1}{N} \})}{\delta} \\
 \ \ = & 0.
 \ \end{aligned}
 \]
This implies
 \[
 \ \underset{\delta\rightarrow 0}{\overline \lim}\dfrac{\gamma(\D(\lambda, \delta) \cap \mathcal E_N)}{\delta} \ge \underset{\delta\rightarrow 0}{\overline \lim}\dfrac{\gamma(\D(\lambda, \delta) \cap U_{\Gamma_1}\cap \mathcal E_N)}{\delta} > 0.
\] 
Hence, $\lambda$ satisfies \eqref{FCEq4}.

Finally, \eqref{FCEq4} does not depend on choices of $\{\Gamma_n\}$, therefore, $\mathcal F$ and $\mathcal R$ are independent of choices of $\{\Gamma_n\}$ up to a set of zero analytic capacity.  
\end{proof}

\begin{example}\label{FRExample} 
Let $K = \overline{\D} \setminus \cup_{n=1}^\infty \D(\lambda_n,\delta_n)$, where $\D = \D(0,1)$, $\D(\lambda_n,\delta_n) \subset \D$, $\overline{\D(\lambda_n,\delta_n)} \cap \overline{\D(\lambda_m,\delta_m)} = \emptyset$ for $n \ne m$, and $\sum \delta_n < \infty$. Set $\partial_e K = \partial \D \cup \cup_{n=1}^\infty \partial \D(\lambda_n,\delta_n)$ (the exterior boundary of $K$).
If $\mu$ is the sum of the arclength measures of the unit circle and all small circles $\partial \D(\lambda_n,\delta_n),$ then $S_\mu$ on $R^t(K, \mu)$ is pure, 
 $\mathcal F_0 = \C\setminus K,~ \mathcal F_- = \partial_e K,~
 \mathcal R_0 = K\setminus \partial_e K,\text{ and }\mathcal F_+ = \mathcal R_1 = \emptyset$
(if $K$ has no interior, then $K$ is a Swiss cheese set).  
\end{example}

\begin{proof}
Let $dz$ be the usual complex measure on $\partial \D$ (counter clockwise) and on each $\partial \D(\lambda_n,\delta_n)$ (clockwise). Then $\int r(z) d z = 0$ for $r\in Rat(K)$. If $g = \frac{dz}{d\mu}$, then $g\perp R^t(K, \mu)$. $S_\mu$ is pure as $|g| > 0, ~\mu-a.a.$.

Let $g_n = g\chi_{\partial \D \cup \cup_{k=1}^n \partial \D(\lambda_k,\delta_k)},$ where $\chi_A$ denotes the characteristic function for a subset $A.$ Then $\|g_n-g\|_{L^1(\mu)} \rightarrow 0$ and $\frac{1}{2\pi I}\mathcal C(g_n\mu)(\lambda) = 1$ for $\lambda \in \D \setminus \cup_{k=1}^n \overline{\D(\lambda_k,\delta_k)}.$ 
Using Lemma \ref{ConvergeLemma}, we have $\mathcal C(g_n\mu)(\lambda) \rightarrow \mathcal C(g\mu)(\lambda),~\gamma-a.a..$ Therefore,  the principal value $\mathcal C(g\mu)(\lambda) = 2\pi I,~\lambda\in K \setminus \partial_e K,~\gamma-a.a..$ This implies  that $K \setminus \partial_e K \subset \mathcal R_0$. It is clear that $\partial_e K \subset \mathcal F_-$ since $\mathcal C (g\mu) (z) = 0$ for $g\perp R^t(K,\mu)$ and $z\in \C \setminus K.$ This completes the proof.
\end{proof}

\section{\textbf{Full Analytic Capacity Density for $\mathcal R$}}

The aim of this section is to prove the following full analytic capacity density property for $\mathcal R,$ which is important for us to characterize $H^\i(\mathcal R)$ in next section.

\begin{theorem}\label{FACDensityThm}
There is $\mathcal Q \subset \C$ with $\gamma(\mathcal Q) = 0$ such that for $\lambda \in \mathcal R \setminus \mathcal Q,$ we have
\begin{eqnarray}\label{FACDensityThmEq}
\ \lim_{\delta\rightarrow 0}\dfrac{\gamma(\D(\lambda, \delta) \cap \mathcal F)}{\delta} = \lim_{\delta\rightarrow 0}\dfrac{\gamma(\D(\lambda, \delta) \setminus \mathcal R)}{\delta} = 0. 
\end{eqnarray}	
\end{theorem}

It is straightforward to prove Theorem \ref{FACDensityThm} if $\mathcal {ND}(\mu) \approx \emptyset,~\gamma-a.a..$ Let us provide a proof for this simple case below.

\begin{proof} (Theorem \ref{FACDensityThm} assuming $\mathcal {ND}(\mu) \approx \emptyset,~\gamma-a.a.$) In this case, $\mathcal F \approx \mathcal F_0,~\gamma-a.a.$ and $\mathcal R \approx \mathcal R_0,~\gamma-a.a..$	By Corollary \ref{ZeroAC}, there exists $\mathcal Q\subset \C$ with $\gamma(\mathcal Q) = 0$ such that for $\lambda\in \mathcal R\setminus \mathcal Q$ and $j\ge 1,$ $\lim_{\epsilon \rightarrow 0} \CT _\epsilon (g_j\mu)(\lambda) = \CT (g_j\mu)(\lambda)$ exists and $\Theta_{g_j\mu} (\lambda) = 0.$ Then, by Lemma \ref{CauchyTLemma}, $\CT (g_j\mu)(z)$ is $\gamma$-continuous at $\lambda.$ For $\lambda \in \mathcal R\setminus \mathcal Q,$ there exists $j_0$ such that $\CT (g_{j_0}\mu)(\lambda) \ne 0.$ Now the proof follows from the inclusion below
\begin{eqnarray} \label{FACDensityThmEq1}
\ \mathcal F \subset \left \{z:~ |\CT (g_{j_0}\mu)(z) - \CT (g_{j_0}\mu)(\lambda)| > \frac{|\CT (g_{j_0}\mu)(\lambda)|}{2} \right \}.
\end{eqnarray} 
\end{proof}

In the case when $\gamma(\mathcal {ND} (\mu)) > 0,$ \eqref{FACDensityThmEq1} may not hold as
$|\CT (g_j\mu)|$ may not be small on $E_1 \subset \mathcal F_+ \cup \mathcal F_-$ with $\gamma (E_1) > 0.$ 
However, we will show in the next lemmas that there exist $\gamma$ comparable subsets $E_N$ near $E_1$ such that $|\CT (g_j\mu)|$ is small on $E_N.$

\begin{lemma} \label{acLemma}
Let $g\perp R^t(K,\mu)$ and let $E\subset \mathcal F_+$ (or $E\subset \mathcal F_-$) be a compact subset with $\gamma (E) > 0$. Then there exists a sequence of compact subsets $\{E_N\}_{N=1}^\infty$ such that $|\CT(g\mu)(z)| \le N^{-1},~ z\in E_N,$ $\sup_{x\in E_N}\text{dist}(x, E) < \epsilon$ for a given $\epsilon > 0,$ and $\gamma (E_N) \gtrsim \gamma (E).$ 
\end{lemma}

\begin{proof} From Theorem \ref{TolsaTheorem} (1) and Proposition \ref{GammaPlusThm} (2), we find $\eta \in M_0^+(E)$ with $1$-linear growth such that $N_2(\eta) \le 1$ and $\gamma (E) \lesssim \|\eta\|$.
 Then
\[
\ \bigcup_{n=1}^\i \Gamma _n = \bigcup_{k=1}^{36}\bigcup_{(k-1)\frac{\pi}{18} \le \beta_n < k\frac{\pi}{18}} \Gamma_n.
\] 
Hence,
 \[
\ \eta(E) \le \sum_{k=1}^{36} \lim_{M\rightarrow \i} \eta \left (\bigcup_{(k-1)\frac{\pi}{18} \le \beta_n < k\frac{\pi}{18}, ~ n\le M} \Gamma_n \cap E \right).
\]
We can find a $k$ and $M$, without loss of generality, assuming $k = 0$, such that 
 \[
 \ G:= \bigcup_{0 \le \beta_n < \frac{\pi}{18}, ~ n\le M} \Gamma_n \cap E
 \]
 satisfying 
$\eta(G) \ge \dfrac{1}{72} \eta(E) \gtrsim \dfrac{1}{72}\gamma(E).$
Hence, we assume $E$ satisfies the following:
\newline
(A) Corresponding Lipschitz function $A_n$ of $\Gamma_n$ satisfying $\|A_n'\|_\infty \le \frac{1}{4}$ (see the proof of \cite[Lemma 4.11]{Tol14});
\newline
(B) The rotation angles $\beta_n$ of $\Gamma_n$ are between $0$ and $\frac{\pi}{18}$;
\newline
(C) $E \subset \cup_{n = 1}^M \Gamma_n$ and $\eta(E) \gtrsim \dfrac{1}{144}\gamma(E)$; and
\newline
(D) We fix an open subset $O \supset E$ such that $\mathcal H^1(O_\Gamma) \le 2\mathcal H^1(E),$ where $O_\Gamma = \cup_{n=1}^M \Gamma_n \cap O.$ 

Claim: There exists $\epsilon_1 > 0 $ (depending on $N$) such that for $\epsilon_2 < \epsilon_1,$ there exists $E_1 \subset E$ with $\eta (E_1) \ge \frac{3}{4} \eta (E)$ satisfying ($I = \sqrt{-1}$): 
 \begin{eqnarray}\label{ClaimEq}
 \ |\CT (g\mu)(z)| \le N^{-1},~ z\in E_N := E_1 + \epsilon_2 I.
 \end{eqnarray}
 
 Proof of the claim:
Set $U(\lambda,\delta) = U_\Gamma \cap \D(\lambda,\delta).$
By Proposition \ref{NZGammaProp} and Theorem \ref{GPTheorem1}, for $\lambda \in E\cap \Gamma_l$ ($1 \le l \le M$), we have
 \begin{eqnarray}\label{ClaimEq0}
 \ \lim_{\delta\rightarrow 0}\dfrac{\gamma(U(\lambda,\delta)\cap \{|\mathcal{C}(g\mu)(z )| > N^{-1} \})}{\delta} = 0.  
\end{eqnarray}
 
Let $B_n$ be a subset consisting of $\lambda\in E$ satisfying
 \begin{eqnarray}\label{ClaimEq1}
 \ \gamma(U(\lambda,\delta)\cap\{|\mathcal{C}(g\mu)(z )| > N^{-1} \}) \le \dfrac{\eta (E) \delta}{80C_T^2\mathcal H^1(E)}\text{ for }\delta \le \frac{1}{n}.
 \end{eqnarray}
We require $n > n_0,$ where $\frac{1}{n_0}$ is less than the smallest distance between $\C\setminus O$ and $E.$ The sets $B_n \subset B_{n+1}$ and
 by \eqref{ClaimEq0}, we see that   
 $\eta (E\setminus \cup_{n=n_0}^\infty B_n) = 0.$ 
 Choose $m > n_0$ large enough such that there is a compact subset $F_m\subset B_m$ with $\eta (F_m) \ge \frac{7}{8}\eta(E).$ For $\delta < \frac 1m,$ using $5r$-covering theorem (see \cite[Theorem 2.2]{Tol14}), there exists a sequence $\{\lambda_k\}_{k=1}^{M_\delta} \subset F_m$ such that 
 \[
 \ F_m \subset \cup_{k=1}^{M_\delta} \D(\lambda_k, \delta) \text{ and } \D(\lambda_{k_1}, \frac 15 \delta) \cap \D(\lambda_{k_2}, \frac 15 \delta) = \emptyset \text{ for } k_1 \ne k_2. 
 \]
 From (D), we see that $M_\delta \delta \le 5 \mathcal H^1(E).$ 
 Set $U_\delta = \cup_{k=1}^{M_\delta} U(\lambda_k, \delta)$ and $V_\delta = U_\delta \cap\{|\mathcal{C}(g\mu)(z )| > N^{-1} \}.$ Applying \eqref{ClaimEq1} and Theorem \ref{TolsaTheorem} (2), we get
 \begin{eqnarray}\label{ClaimEq2}
 \ \gamma(V_\delta) \le C_T\sum_{k=1}^{M_\delta} \gamma(U(\lambda_k,\delta)\cap\{|\mathcal{C}(g\mu)(z )| > N^{-1} \}) \le \frac{\eta(E)}{16C_T}.
 \end{eqnarray}
By (A) and (B), we see that there exists $\epsilon_1 > 0$ such that if $\epsilon_2 < \epsilon_1$ and $L_m = F_m + \epsilon_2 I,$ then $L_m \subset U_\delta.$
Let 
 \[
 \ E_1 = L_m \cap \{|\mathcal{C}(g\mu)(z )| \le N^{-1} \}  - \epsilon_2 I\text{ and } E_0 = L_m \cap \{|\mathcal{C}(g\mu)(z )| > N^{-1} \}  - \epsilon_2 I.
 \]
Then
 \[
\ \eta ( E_1 ) =  \eta \left ( (L_m  - \epsilon_2 I) \setminus E_0 \right ) \ge \eta ( F_m) - \eta ( E_0 ).
\]
$\eta |_{E_0}$ is of $1$-linear growth and $N_2(\eta |_{E_0}) \le 1.$
From \eqref{ClaimEq2}, Theorem \ref{TolsaTheorem} (1), and Proposition \ref{GammaPlusThm} (3), we get
 \[
 \  \eta (  E_0 ) \le 2C_T \gamma (E_0) =  2C_T\gamma \left (  L_m \cap \{|\mathcal{C}(g\mu)(z )| > N^{-1} \} \right ) \le 2C_T \gamma (V_\delta) \le \dfrac{1}{8} \eta (E).
\]
Combining above two inequalities, we get $\eta (E_1) \ge \frac{3}{4} \eta (E)$ and \eqref{ClaimEq} holds. This completes the proof of the claim. The lemma now follows from Theorem \ref{TolsaTheorem} (1), Proposition \ref{GammaPlusThm} (3), and the claim.
\end{proof}

\begin{corollary}\label{acCorollary}
Let $g\perp R^t(K,\mu)$ and let $E\subset \mathcal F$ be a compact subset with $\gamma (E) > 0$. Then there exists a sequence of compact subsets $\{E_N\}_{N=1}^\infty$ such that $|\CT(g\mu)(z)| \le \frac 1N,~ z\in E_N,$ $\sup_{x\in E_N}\text{dist}(x, E) < \epsilon$ for a given $\epsilon > 0,$ and $\gamma (E_N) \gtrsim \gamma (E).$ 
\end{corollary}

\begin{proof}
If $E \subset \mathcal F_0,$ then by Proposition \ref{NF0Prop} (2), we have $\CT(g\mu)(z) = 0, ~ \gamma|_E-a.a..$ So we can choose $E_N = E$ in this case. In general, by Theorem \ref{TolsaTheorem} (2), we have
\[
\ \gamma(E) \le C_T (\gamma(E\cap \mathcal F_0) + \gamma(E\cap \mathcal F_+) + \gamma(E\cap \mathcal F_-)).
\]
The proof now follows from Lemma \ref{acLemma}.
\end{proof}

The following lemma is a generalization of \eqref{FACDensityThmEq1} when $\gamma(\mathcal {ND} (\mu)) > 0.$

\begin{lemma} \label{GLemma}
Let $O$ be an open subset of $\mathbb C$, $g\perp \rtkmu$, and $a\ne 0$. If for some $0 < \epsilon < \frac{|a|}{2}$ and $\lambda\in \mathbb C$,
 \begin{eqnarray}\label{GLemmaAssump}
 \ \lim_{\delta\rightarrow 0 }\dfrac{\gamma(\D(\lambda, \delta)\cap O \cap \{|\mathcal C(g\mu)(z) - a| > \epsilon\})}{\delta} = 0,
 \end{eqnarray} 
then
 \[
 \ \lim_{\delta\rightarrow 0 }\dfrac{\gamma(\D(\lambda, \delta)\cap O \cap \mathcal F)}{\delta} = 0.
 \] 
\end{lemma}

\begin{proof}
Suppose that there exists $\epsilon_0 > 0$  and $\delta_n \rightarrow 0$ such that 
	\[
 \ \gamma(\D(\lambda, \delta_n)\cap O\cap \mathcal F)\ge 2\epsilon_0 \delta_n.
 \]
Let $F_n \subset \D(\lambda, \delta_n)\cap O\cap \mathcal F$ be a compact subset such that $\gamma(F_n)\ge \epsilon_0 \delta_n$.
Let $d_n$ be the smallest distance between $(\D(\lambda, \delta_n)\cap O)^c$ and $F_n$. From
Corollary \ref{acCorollary}, we find a compact subset $E_N^n\subset \D(\lambda, \delta_n)\cap O$ such that $\gamma(E_N^n) \gtrsim \gamma(F_n),$ 
 $\sup_{x\in E_N^n}\text{dist} (x, F_n) < d_n,$ and $|\CT(g\mu)(z)| < \frac{|a|}{2}, ~ z \in E_N^n.$
Hence, 
 \[
 \ E_N^n \subset \D(\lambda, \delta_n)\cap O \cap \{|\mathcal C(g\mu)(z) - a| > \epsilon\}.
 \]
Hence,
 \[
 \ \dfrac{\gamma(\D(\lambda, \delta_n)\cap O\cap \{|\mathcal C(g\mu)(z) - a| > \epsilon\})}{\delta_n} \ge \dfrac{\gamma(E_N^n)}{\delta_n} \gtrsim \epsilon _0,
 \]
which contradicts the assumption \eqref{GLemmaAssump}. The lemma is proved.  
\end{proof}

\begin{proof} (Theorem \ref{FACDensityThm})
For almost all $\lambda\in \mathcal R_0$ with respect to $\gamma$, there exists $j_0$ such that $\mathcal C(g_{j_0}\mu)(\lambda) \ne 0$ exists and $\lambda \in \mathcal {ZD}(g_{j_0}\mu)$. \eqref{FACDensityThmEq} follows from Lemma \ref{CauchyTLemma} and Lemma \ref{GLemma} for $O = \mathbb C.$

For $\lambda\in \mathcal R_1\cap \Gamma_1,~\gamma-a.a.$, there are integers $j_1$ and $j_2$ such that:  
\begin{eqnarray}\label{DensityTheoremEqa}
\ h_i(\lambda) \ne 0\text{ and }\lim_{\delta\rightarrow 0 }\dfrac{\gamma(\D(\lambda, \delta)\cap \Gamma_1 \cap \{|h_i(z)-h_i(\lambda)| > \epsilon \})}{\delta} = 0
\end{eqnarray}
for $i = 1,2,3,$ where $h_1 = h,$ $h_2 = v^+(g_{j_1}\mu, \Gamma_1, \beta_1),$ and $h_3 = v^-(g_{j_2}\mu, \Gamma_1, \beta_1)$ (notice $\gamma |_{\Gamma_1} \approx \mathcal H^1 |_{\Gamma_1}$ by \eqref{HACEq});
\begin{eqnarray}\label{DensityTheoremEqb}
\ \lim_{\delta\rightarrow 0 }\dfrac{\gamma(\D(\lambda, \delta)\cap U_{\Gamma_1} \cap \{|\CT(g_{j_1}\mu)(z) -h_2(\lambda)| > \epsilon \})}{\delta} = 0
\end{eqnarray}
by Theorem \ref{GPTheorem1} (b);
 and 
\begin{eqnarray}\label{DensityTheoremEqc}
\ \lim_{\delta\rightarrow 0 }\dfrac{\gamma(\D(\lambda, \delta)\cap L_{\Gamma_1} \cap \{|\CT(g_{j_2}\mu)(z) - h_3(\lambda)| > \epsilon \})}{\delta} = 0
\end{eqnarray}
by Theorem \ref{GPTheorem1} (c).
  For $\epsilon < \frac 12 \min (|h_1(\lambda)|,|h_2(\lambda)|,|h_3(\lambda)|),$ since $\mathcal F_0 \cap \mathcal N(h) = \emptyset$ (see Theorem \ref{NDecompThm}), we get 
 \[
 \begin{aligned}
 \ & \Gamma_1\cap \mathcal F_0\subset\Gamma_1\cap\{|h_1(z)- h_1(\lambda)| > \epsilon \}, \\
 \ & \Gamma_1\cap \mathcal F_+\subset\Gamma_1\cap\{|h_2(z)- h_2(\lambda)| > \epsilon \}, \\
 \ & \Gamma_1\cap \mathcal F_-\subset\Gamma_1\cap\{|h_3(z)- h_3(\lambda)| > \epsilon \}.
 \end{aligned}
 \]
 Using Theorem \ref{TolsaTheorem} (2) and \eqref{DensityTheoremEqa}, we have
 \begin{eqnarray}\label{DensityTheoremEq2}
 \ \lim_{\delta\rightarrow 0 }\dfrac{\gamma(\D(\lambda, \delta)\cap \Gamma_1 \cap \mathcal F)}{\delta} = 0.
\end{eqnarray}
 By \eqref{DensityTheoremEqb}, applying Lemma \ref{GLemma} for $O= U_{\Gamma_1}$, we have 
 \begin{eqnarray}\label{DensityTheoremEq3}
 \ \lim_{\delta\rightarrow 0 }\dfrac{\gamma(\D(\lambda, \delta)\cap U_{\Gamma_1} \cap \mathcal F)}{\delta} = 0.
\end{eqnarray}
 By \eqref{DensityTheoremEqc}, applying Lemma \ref{GLemma} for $O= L_{\Gamma_1}$, we have
 \begin{eqnarray}\label{DensityTheoremEq4}
 \ \lim_{\delta\rightarrow 0 }\dfrac{\gamma(\D(\lambda, \delta)\cap L_{\Gamma_1} \cap \mathcal F)}{\delta} = 0.
 \end{eqnarray} 
 Using Theorem \ref{TolsaTheorem} (2) for \eqref{DensityTheoremEq2}, \eqref{DensityTheoremEq3}, and \eqref{DensityTheoremEq4}, we finish the proof.  
\end{proof}

 \section{\textbf{Surjectivity of the Map $\rho$}}
 
 \begin{definition}\label{HEDAlgDef}
Let $\mathcal D \subset \C$ be a bounded Borel subset. Let $H(\mathcal D)$ be the set of functions $f(z)$ such that $f(z)$ is bounded and analytic on $\mathbb C\setminus E_f$ for some compact subset $E_f\subset \mathbb C \setminus \mathcal D.$
Define $H^\i (\mathcal D)$ to be the weak-star closed subalgebra of $L^\i (\area _{\mathcal D})$ generated by functions in $H(\mathcal D).$ 
\end{definition}

If $\mathcal D$ is a bounded open subset, then $H^\i(\mathcal D)$ is the algebra of bounded and analytic functions on $\mathcal D.$ If $K$ is a Swiss cheese set as in Example \ref{FRExample}, then $\mathcal R = K \setminus \partial_e K$ and $H^\i(\mathcal R)$ is different from the bounded and analytic functions on an open subset as $\mathcal R$ has no interior points. 

\begin{definition} \label{GammaOpenDef}
 The subset $\mathcal D \subset \C$ is \index{$\gamma$-open}{\em $\gamma$-open} if
 there exists a subset $\mathcal Q$ with $\area(\mathcal Q) = 0$ such that for $\lambda \in \mathcal D \setminus \mathcal Q$,
 \begin{eqnarray}\label{DDensityDef}
  \ \lim_{\delta\rightarrow 0} \dfrac{\gamma(\D(\lambda, \delta)\setminus \mathcal D)}{\delta} = 0.
 \end{eqnarray}
 If $\gamma(\mathcal Q) = 0,$ then $\mathcal D$ is called \index{strong $\gamma$-open}{\em strong  $\gamma$-open}.
 \end{definition}
 
 Clearly, by \eqref{AreaGammaEq}, if $\mathcal D$ is strong $\gamma$-open, then $\mathcal D$ is $\gamma$-open. By Theorem \ref{FACDensityThm}, we see that $\mathcal R$ is strong $\gamma$-open.
 
 \begin{definition}\label{BBDef}
 For $\delta > 0$ and integers $i,j,$  let $c_{ij} = (\frac{i+1}{2}\delta, \frac{j+1}{2}\delta).$ We say $\{\delta, E_{ij}, \eta_{ij}, \tilde f_{ij}, k_1\}$ ($k_1 \ge 1$) is a building block for $\mathcal D$ if $E_{ij} \subset \D(c_{ij}, k_1 \delta) \setminus \mathcal D$ is a compact subset, $\eta_{ij} \in M_0^+(E_{ij})$ satisfies $\|\mathcal C_\epsilon (\eta_{ij})\| \lesssim 1$ and $\|\eta_{ij}\| = \gamma (\D(c_{ij}, k_1 \delta) \setminus \mathcal D),$ and $\tilde f_{ij}\in L^\i (\mu)$ satisfies $\|\tilde f_{ij}\|_{L^\i (\mu)} \lesssim 1$ and $\tilde f_{ij} (z) = \mathcal C (\eta_{ij})(z)$ for $z \in \C \setminus E_{ij}.$
 \end{definition}
 
 By Theorem \ref{TolsaTheorem} (1) and Proposition \ref{GammaPlusThm} (1), we see that for $\delta > 0,$ there always exists a building block $\{\delta, E_{ij}, \eta_{ij}, \tilde f_{ij}, k_1\}.$	 If $\mathcal D$ is bounded, then there are only finite many $i$ and $j$ such that $\eta_{ij} \ne 0.$

 \begin{theorem}\label{HDAlgTheorem}
Let $\mathcal D$ be a $\gamma$-open bounded Borel subset. Let $f\in L^\i (\area_{\mathcal D})$ be given with $\|f\|_{L^\infty (\area_{\mathcal D})}\le 1$. 
Then the following statements are equivalent.

(1) There exists $C_f>0$ (depending on $f$) such that for all $\lambda\in \mathbb C$, $\delta > 0$, and for all choices of a smooth non-negative function $\varphi$ with support in $\D(\lambda, \delta)$ satisfying $\varphi (z) \le 1$  and $\left \|\frac{\partial \varphi (z)}{\partial  \bar z} \right \|_\infty \lesssim \frac{1}{\delta},$ we have   
 \begin{eqnarray}\label{HDAlgTheoremEq}
 \ \left | \int (z-\lambda)^nf(z) \dfrac{\partial \varphi (z)}{\partial  \bar z} d\area_{\mathcal D}(z) \right | \le C_f\delta^n \gamma (\D(\lambda, \delta) \setminus \mathcal D) \text{ for } n \ge 0.
 \end{eqnarray} 
 
 (2) Let $\{\delta, E_{ij}, \eta_{ij}, \tilde f_{ij},k_1\}$ be a building block for $\mathcal D$ as in Definition \ref{BBDef}. 
 Then there exists $f_\delta\in H(\mathcal D)\cap L^\i (\mu)$ that is a finite linear combination of $\tilde f_{ij}$ such that $\|f_\delta\|_{\mathcal D},~ \|f_\delta\|_{L^\i (\mu)} \lesssim C_f$ and there exists a subsequence  $\{f_{\delta_m}\}$ satisfying $f_{\delta_m}(z) \rightarrow f(z),~\area_{\mathcal D}-a.a.$ as $\delta_m\rightarrow 0.$ 
 
 (3) $f\in H^\i(\mathcal D).$   
\end{theorem}

\begin{proof}
(1)$\Rightarrow$(2): See \cite[Theorem 4.3]{y23}.
 The proof is technical and requires to use the modified Vitushkin scheme of Paramonov \cite{p95} (see \cite[sections 3 \& 4]{y23}). 
 
 (2)$\Rightarrow$(3) is trivial.
 
 (3)$\Rightarrow$(1): The Vitushkin's localization operator $T_\varphi$ is defined by
 \[
 \ (T_\varphi f)(\lambda) = \dfrac{1}{\pi}\int \dfrac{f(z) - f(\lambda)}{z - \lambda} \bar\partial \varphi (z) d\area(z),
 \]
where $f\in L^1_{loc} (\C)$.	 Clearly, $(T_\varphi f)(z) = - \frac{1}{\pi}\mathcal C(\varphi \bar\partial f\area ) (z).$
Therefore, by \eqref{CTDistributionEq}, $T_\varphi f$ is analytic outside of $\text{supp} (\bar \partial f) \cap\text{supp} (\varphi).$ If $\text{supp} (\varphi) \subset \D(a,\delta),$ then $\| T_\varphi f\|_\i   \le  4\|f\|_\i  \delta\|\bar\partial \varphi\|.$
 See \cite[VIII.7.1]{gamelin} for the details of $T_\varphi.$

Let $f\in H^\infty (\mathcal D)$ with $\|f\|_{\mathcal D} \le 1$ and $f(z) = 0,~ z\in \C \setminus \mathcal D.$ 
 Let $\{f_m\}\subset H(\mathcal D),$ where $f_m$ is bounded and analytic on $\C\setminus E_m$ for some compact subset $E_m \subset \C \setminus \mathcal D,$ such that $f_m$ converges to $f$ in $L^\i(\area_{\mathcal D})$ weak-star topology. Then $\|f_m\|_{\mathcal D} \le C_f$ for some constant $C_f.$ We may assume $f_m(z)\rightarrow f(z),~\area_{\mathcal D}-a.a.$ and $\|f_m\|_{\C} \le 2C_f.$ 
 The function $T_\varphi f_m$ is analytic on $\C_\i \setminus (\text{supp}(\varphi)\cap E_m).$ Therefore, for $n \ge 0,$
 \[ 
\begin{aligned}
\ & \left | \int_{\mathcal D} f_m(z) (z - \lambda)^{n}\bar \partial \varphi (z) d\area (z)\right | \\
\ \le & \left | \int_{\C \setminus \mathcal D} f_m(z) (z - \lambda)^{n}\bar \partial \varphi (z) d\area (z)\right | + \left | \int f_m(z) (z - \lambda)^{n}\bar \partial \varphi (z) d\area (z)\right | \\
\ \lesssim & C_f \delta ^{n-1} \area (\D(\lambda, \delta)\setminus \mathcal D) + \pi \|T_{\varphi} ((z - \lambda)^{n}f_m)\| \gamma (\text{supp}(\varphi)\cap E_m) \\
\ \lesssim & C_f \delta ^{n} \gamma (\D(\lambda, \delta)\setminus \mathcal D),
\end{aligned}
\]
where \eqref{AreaGammaEq} is used in the last step.
Taking $m\rightarrow \i,$ we prove \eqref{HDAlgTheoremEq}. 
 \end{proof}
 
 The aim of this section is to use Theorem \ref{HDAlgTheorem} to prove the following Lemma.
 
 \begin{lemma}\label{RhoOntoLemma}
 For $f\in H^\infty(\mathcal R)$ with $\|f\|_{\mathcal R} \le 1,$ there exists $\tilde f\in R^{t,\i}(K, \mu)$ such that $\rho(\tilde f) = f.$ Moreover, $\|\tilde f\|_{L^\i(\mu)} \lesssim C_f,$ where $C_f$ is the constant in \eqref{HDAlgTheoremEq}. Consequently, $\rho$ is surjective.
 \end{lemma}
  
To prove Lemma \ref{RhoOntoLemma}, we need a couple of lemmas.

\begin{lemma}\label{ACIncreasing}
Let $E_n\subset E_{n+1}\subset \D(0, R)$ be a sequence of subsets. Then
\[
\ \gamma \left (\cup_{n = 1}^\i E_n \right ) \lesssim \lim_{n\rightarrow \i} \gamma(E_n).
\]
\end{lemma}

\begin{proof}
By Theorem \ref{TolsaTheorem} (1) and Proposition \ref{GammaPlusThm} (2), there exists a compact subset $F\subset \cup_{n = 1}^\i E_n$ and $\eta \in M_0^+(F)$ with $1$-linear growth such that $N_2(\eta) \le 1$ and 
 \[ 
 \ \gamma \left (\cup_{n = 1}^\i E_n \right ) \lesssim \|\eta\| \lesssim \lim_{n\rightarrow \i} \eta(E_n). 
 \]
 Since $N_2(\eta|_{E_n}) \le 1,$
 by Proposition \ref{GammaPlusThm} (3) and Theorem \ref{TolsaTheorem} (1), we get $\eta(E_n) \lesssim \gamma(E_n).$
\end{proof}

\begin{lemma} \label{BBFRLambda} 
Let $E_1\subset \mathcal F$ be a compact subset with $\gamma(E_1) > 0$. Then 
there exists $f\in R^{t,\i}(K, \mu)$ and $\eta\in M_0^+(E_1)$ such that $\|\mathcal C_\epsilon (\eta) \| \lesssim 1,$ $f(z) = \mathcal C(\eta)(z)$ for $z\in \C_\i \setminus \text{spt}\eta$, 
 \[
 \ \|f\|_{L^\infty(\mu)} \lesssim 1,~ f(\infty) = 0,~ f'(\infty) = - \gamma(E_1),
 \]
and
 \begin{eqnarray}\label{BBFRLambdaEq1}
 \ \mathcal C(\eta)(z) \mathcal C(g_j\mu) (z) = \mathcal C(fg_j\mu) (z), ~\gamma|_{\C \setminus \text{spt}\eta}-a.a. \text{ for }j \ge 1.
 \end{eqnarray} 
\end{lemma}

\begin{proof}
Let the measure $\eta_1$ and the function $f_1$ be constructed as in Lemma \ref{BBFunctLemma} for $\{g_j\}$ and $E_1.$
From Theorem \ref{TolsaTheorem} (2), we have
\[
\ \max(\gamma (E_1 \cap \mathcal F_0 ),~ \gamma (E_1 \cap \mathcal F_+),~ \gamma (E_1 \cap \mathcal F_- )) \ge \dfrac{1}{3C_T}\gamma (E_1).
 \]
 Therefore, we shall consider the following three cases.
 
 Case I (assuming $E_1 \subset \mathcal F_0 $):  From \eqref{BBFunctLemmaEq1} and \eqref{BBFunctLemmaEq2}, we see that $f_1\in R^{t,\i}(K, \mu)$ and $f_1(\lambda)\mathcal C(g_j\mu) (\lambda) = \mathcal C(f_1g_j\mu) (\lambda,~\gamma |_{\C \setminus \text{spt}\eta_1}-a.a.$ for $j \ge 1.$
 Set 
 \[
 \ f = \dfrac{f_1}{\|\eta_1\|} \gamma (E_1) \text{ and } \eta  = \dfrac{\eta_1}{\|\eta_1\|} \gamma (E_1). 
 \]
 
 Case II (assuming $E_1 \subset \mathcal F_+$): Using Lemma \ref{ACIncreasing}, we assume that there exists a positive integer $n_0$ such that 
$E_1\subset \bigcup_{n=1}^{n_0}\Gamma_n.$
 Put $E_1 = \cup_{n=1}^{n_0}F_n,$ where $F_n \subset \Gamma_n$ and $F_n \cap F_m = \emptyset$ for $n \ne m.$ Since $\eta_1$ is of $1$-linear growth, we can set $\eta_1 = \sum_{n=1}^{n_0} w_n(z)\mathcal H^1 |_{\Gamma_n},$ where $w_n$ is supported on $F_n.$ Define 
 \[
 \ f_2(z) = f_1(z) - \frac{1}{2} \sum_{n=1}^{n_0} e^{-i\beta_n}L((z - z_n)e^{-i\beta_n} + z_n)^{-1}w_n(z).
 \] 
 Then using Theorem \ref{GPTheorem1}, we conclude that
  $f_2,w_n\in L^\infty(\mu)$ and $\|f_2\|_{L^\infty(\mu)} \lesssim 1.$
 From \eqref{BBFunctLemmaEq1}, we get
 \begin{eqnarray}\label{BBFGEq3} 
 \begin{aligned}
 \ \int f_2 g_jd\mu = &  - \int \mathcal C(g_j\mu)  d\eta_1 - \frac{1}{2} \sum_{n=1}^{n_0} \int_{F_n} e^{-i\beta_n}L((z - z_n)e^{-i\beta_n} + z_n)^{-1}g_j h d\eta_1 \\
 \    = & - \sum_{n=1}^{n_0} \int_{F_n} v^+(g_j\mu, \Gamma_n, \beta_n) d\eta_1.
 \end{aligned}
 \end{eqnarray}
 Similarly, for $\lambda\in \C \setminus \text{spt}\eta_1$ and by \eqref{BBFunctLemmaEq2}, we have
 \begin{eqnarray}\label{BBFGEq4} 
 \ \int \dfrac{f_2(z)-f_2(\lambda)}{z-\lambda} g_j(z)d\mu (z) = - \sum_{n=1}^{n_0} \int_{F_n} v^+(g_j\mu, \Gamma_n, \beta_n)(z)\dfrac{d\eta_1(z)}{z-\lambda} .
 \end{eqnarray}
 Since $v^+(g_j\mu, \Gamma_n, \beta_n)(z) = 0, ~z \in F_n,~ \mathcal H^1 |_{\Gamma_n}-a.a.,$ by \eqref{BBFGEq3}, we get 
 $f_2\in R^{t,\i}(K, \mu).$
 Similarly, from \eqref{BBFGEq4}, we see that $\frac{f_2(z)-f_2(\lambda)}{z-\lambda}\in R^{t,\i}(K, \mu)$ and  $f_2(\lambda)\mathcal C(g_j\mu) (\lambda) = \mathcal C((f_2 g_j\mu) (\lambda)$ for $\lambda\in \C \setminus \text{spt}\eta_1,~\gamma-a.a.$ 
  Set 
 \[
 \ f = \dfrac{f_2}{\|\eta_1\|} \gamma (E_1) \text{ and } \eta  = \dfrac{\eta_1}{\|\eta_1\|} \gamma (E_1). 
 \]
 
 Case III (assuming $E_1 \subset \mathcal F_-$): The proof is the same as Case II if we modify the definition of $f_2$ by the following
\[
 \ f_2(z) = f_1(z) + \frac{1}{2} \sum_{n=1}^{n_0} e^{-i\beta_n}L((z - z_n)e^{-i\beta_n} + z_n)^{-1}w_n(z).
 \] 
 
 Then $f$ and $\eta$ satisfy the properties of the lemma. The lemma is proved.
 \end{proof}
 
 We are now ready to prove Lemma \ref{RhoOntoLemma}.
 
 \begin{proof} (Lemma \ref{RhoOntoLemma}):
 Let $f\in H^\infty (\mathcal R)$ with $\|f\|_{\mathcal R} \le 1$ and $f(z) = 0,~ z\in \C \setminus \mathcal R.$ 
 Let $E_{ij} \subset \D(c_{ij}, k_1 \delta)\setminus \mathcal R$ be a compact subset such that $\gamma (E_{ij}) \ge \frac 12 \gamma (\D(c_{ij}, k_1 \delta)\setminus \mathcal R).$ Let  
 $\eta_{ij}\in M_0^+(E_{ij})$ and $\tilde f_{ij}\in R^{t,\i}(K,\mu)$ be as in Lemma \ref{BBFRLambda} such that $\|\CT_\epsilon(\eta_{ij})\| \lesssim 1,$ $\|\eta_{ij}\| = \gamma (\D(c_{ij}, k_1 \delta)\setminus \mathcal R),$ $\|\tilde f_{ij}\| \lesssim 1,$ $\tilde f_{ij}(z) = \CT(\eta_{ij})(z)$ for $z\in \C \setminus E_{ij},$ and 
 \begin{eqnarray}\label{RhoOntoLemmaPfEq1}
 \ \CT(\eta_{ij})(z) \CT (g\mu)(z) = \CT (\tilde f_{ij}g\mu)(z),~ \gamma|_{\C \setminus E_{ij}}-a.a.\text{ for }g\perp \rtkmu.
 \end{eqnarray}

 It is clear that $\{\delta, E_{ij}, \eta_{ij}, \tilde f_{ij}, k_1\}$ is a building block for $\mathcal R$ as in Definition \ref{BBDef}.
 By Theorem \ref{FACDensityThm}, $\mathcal R$ is strong $\gamma$-open. Hence, using Theorem \ref{HDAlgTheorem} (2), we let $f_\delta \in H(\mathcal R)\cap R^{t,\i}(K,\mu)$ be the function that is a finite linear combination of $\tilde f_{ij}$ such that there exists a subsequence $\{f_{\delta_m}\}$ satisfying $\|f_{\delta_m}\|_{\mathcal R},~\|f_{\delta_m}\|_{L^\i(\mu)} \lesssim C_f$ and $f_{\delta_m} (z) \rightarrow f(z),~\area_{\mathcal R}-a.a..$ Therefore, by passing to a subsequence, we may assume that $f_{\delta_m} (z) \rightarrow \tilde f(z)$ in $L^\infty(\mu)$ weak-star topology. Hence, $\tilde f\in R^{t,\i}(K,\mu),$ $\|\tilde f\|_{L^\i(\mu)} \lesssim C_f,$ and $\mathcal C(f_{\delta_m} g\mu)(z) \rightarrow \mathcal C(\tilde fg\mu)(z), ~\area-a.a..$ From \eqref{RhoOntoLemmaPfEq1} and \eqref{acZero},  taking $m\rightarrow \i,$ we infer that
\[
\ f(z)\mathcal C(g\mu)(z) = \mathcal C(\tilde fg\mu)(z), ~\area-a.a.\text{ for }g\perp \rtkmu,
\]
which implies $\rho(\tilde f) = f$ by \eqref{RhoExistLemmaEq1}.
\end{proof}

\section{\textbf{Proof of Theorem \ref{mainThm}}}

We restate Theorem \ref{mainThm} as the following form.

\begin{theorem}\label{mainThmR}
 Let $\mu\in M_0^+(K)$ for a compact set $K\subset \C$.
 Suppose that $1\le t < \infty$ and $S_\mu$ on $R^t(K, \mu)$ is pure. Let $\mathcal F$ and $\mathcal R$ be the non-removable boundary and removable set for $R^t(K, \mu),$ respectively. Let $\rho$ be the map defined as in Lemma \ref{RhoExistLemma}. 
Then $\text{spt}\mu \subset \overline{\mathcal R}$ and $\rho$ is an isometric isomorphism and a weak-star homeomorphism from $R^{t,\i}(K, \mu)$ onto $H^\infty(\mathcal R)$ satisfying
\newline
(1) $\rho(r) = r$ for $r\in\text{Rat}(K),$	
\newline
(2) $\CT(g\mu)(z) = 0,~\area_{\mathcal F}-a.a.$ for $g\perp \rtkmu,$ and
\newline
(3) $\rho(f)(z)\CT(g\mu)(z) = \CT(fg\mu)(z),~ \gamma-a.a.$ for $f\in R^{t,\i}(K, \mu)$ and $g\perp \rtkmu.$
 \end{theorem}

To prove the image of $\rho$ is a subset of $H^\i(\mathcal R),$ by Theorem \ref{HDAlgTheorem}, we need to prove the following lemma.

\begin{lemma} \label{MLemma1}
If $f\in R^{t,\i}(K,\mu)$ and $\varphi$ is a smooth function with support in $\D(\lambda, \delta)$, then
 \begin{eqnarray}\label{MLemma1Eq}
 \ \left |\int \rho(f)(z) \bar \partial \varphi (z) d\area(z) \right | \lesssim \|\rho(f)\| \delta \|\bar \partial \varphi\| \gamma(\D(\lambda, 2\delta) \cap \mathcal F).
 \end{eqnarray}
\end{lemma}

Let us use Lemma \ref{MLemma1} to prove Theorem \ref{mainThmR} before proving the lemma.

\begin{proof} (Theorem \ref{mainThmR} assuming Lemma \ref{MLemma1} holds): 
(1) is trivial. (2) follows from \eqref{acZero}. (3) follows from \eqref{RhoExistLemmaEq1}. For $\D(\lambda,\delta)\cap \overline{\mathcal R} = \emptyset,$ by (2), we have $\CT(g\mu)(z) = 0,~\area |_{\D(\lambda,\delta)}-a.a.$ for $g\perp \rtkmu,$ which implies $\mu (\D(\lambda,\delta)) = 0$ since $S_\mu$ is pure. Hence, $\text{spt}\mu \subset \overline{\mathcal R}.$ It remains to prove that $\rho$ is  an isometric isomorphism and a weak-star homeomorphism. 
 
For $f\in R^{t,\i}(K,\mu),$ by Lemma \ref{MLemma1} and Theorem \ref{HDAlgTheorem}, we conclude that $\rho(f) \in H^\i(\mathcal R)$ since $\mathcal R$ is strong $\gamma$-open by Theorem \ref{FACDensityThm}. Set $F = \rho(f).$ Using Lemma \ref{RhoOntoLemma}, we see that there exists $\tilde F\in R^{t,\i}(K,\mu)$ such that $\rho(\tilde F) = F$
and $\|\tilde F\|_{L^\i(\mu)} \le C_1\|F\|,$ where $C_1 > 0$ is an absolute constant and $C_F \le C_1 \|F\|$ by \eqref{MLemma1Eq} and \eqref{HDAlgTheoremEq}. From \eqref{RhoExistLemmaEq1}, we get
\[
 \ \mathcal C(\tilde Fg\mu)(z) = F(z)\mathcal C(g\mu)(z) = \mathcal C(fg\mu)(z), ~ \area-a.a.
 \]
 for $g\perp R^t(K, \mu),$ which implies $\tilde F = f$ since $S_\mu$ is pure. Hence, 
 \[
 \ \|f\|_{L^\infty(\mu )} \le C_1 \|\rho(f)\|_{L^\infty(\area_{\mathcal R })}.
 \]
Using Proposition \ref{Rhoprop} (1),
 \[
 \ \|f^n\|_{L^\infty(\mu )}^{\frac 1n} \le C_1^{\frac 1n} \|(\rho(f))^n\|^{\frac 1n}_{L^\infty(\area_{\mathcal R })}. 
 \]
Thus, taking $n\rightarrow \i,$ we get $\|f\|_{L^\infty(\mu )} \le \|\rho (f)\|_{L^\infty(\area_{\mathcal R })}$. So, by Proposition \ref{Rhoprop} (2), we have proved that
 \[
 \ \|\rho(f)\|_{L^\infty(\area_{\mathcal R})} = \|f\|_{L^\infty(\mu)}, ~ f\in R^{t,\i}(K, \mu).
 \]
It is clear that the map $\rho$ is injective. Lemma \ref{RhoOntoLemma} implies that  $\rho$ is surjective. Therefore, $\rho$ is bijective      
 isomorphism between two Banach algebras $R^{t,\i}(K, \mu)$ and $H^\infty (\mathcal R)$. Clearly $\rho$ is also a weak-star sequentially continuous, so an application of Krein-Smulian Theorem shows that $\rho$ is a weak-star homeomorphism.	
\end{proof}

To prove Lemma \ref{MLemma1}, we need some lemmas.
Let $\phi$ be a smooth non-negative function on $\mathbb R$ supported on $[0,1]$ with $0 \le \phi(|z|) \le 1$ and $\int \phi(|z|) d\area(z) = 1$.  For $\epsilon > 0,$ define 
$\phi_\epsilon(z)= \frac{1}{\epsilon^2} \phi (\frac{|z|}{\epsilon})$ and 
 $K_\epsilon = - \frac{1}{z} * \phi_\epsilon.$
For $\nu\in M_0(\C),$ define
$\tilde {\mathcal C}_\epsilon \nu = K_\epsilon * \nu.$
The kernel $K_\epsilon$ is a smooth function,  satisfies $\|K_\epsilon \|_\infty \lesssim \frac{1}{\epsilon}$,
$K_\epsilon (z) = - \dfrac{1}{z}\text{ for } |z| \ge \epsilon,$
 $\tilde {\mathcal C}_\epsilon \nu = \phi_\epsilon * \mathcal C\nu = \mathcal C(\phi_\epsilon*\nu),$
and
\begin{eqnarray}\label{CTEEstimate}
 \ |\tilde {\mathcal C}_\epsilon \nu (\lambda) - \mathcal C_\epsilon \nu (\lambda)| \lesssim  \dfrac{|\nu|(\D(\lambda, \epsilon))}{\epsilon} \lesssim \mathcal M_\nu(\lambda).
 \end{eqnarray}
We denote $\mathcal V_\nu$ the set of $z\in \C$ for which $\lim_{\epsilon\rightarrow 0} \CT_\epsilon \nu(z) = \CT \nu(z)$ exists. Set $\mathcal X_\nu = \C \setminus \mathcal V_\nu.$ Then $\gamma (\mathcal X_\nu) = 0$ by Corollary \ref{ZeroAC}. 

\begin{lemma}\label{distributionLemma}
Suppose that $\eta \in M_0^+(\C),$  $\eta$ is of $1$-linear growth, and $\|\mathcal C_\epsilon (\eta)\| \le 1$. If $\nu \in M_0(\C)$ satisfies $|{\mathcal C}_\epsilon (\nu) (\lambda)|, ~ \mathcal M_\nu(\lambda) \le M < \infty, ~ \eta-a.a.$, then there are two functions $F_1\in L^\infty (|\nu|)$ and $F_2\in L^\infty (\eta)$ with $F_1(z) = \mathcal C(\eta)(z),~ \nu |_{\mathcal {ZD}(\eta)  \setminus \mathcal X_\eta}-a.a.$ and $F_2(z) = \mathcal C(\nu)(z),~ \eta |_{\mathcal {ZD}(\eta)} -a.a.$ such that in the sense of distribution,
 \[
 \ \bar\partial (\mathcal C (\eta)\mathcal C (\nu)) = - \pi(F_1\nu + F_2 \eta).
 \]
\end{lemma}

\begin{proof}
 The transform $\tilde {\mathcal C}_\epsilon \nu$ is smooth and $\|\tilde {\mathcal C}_\epsilon \nu - {\mathcal C} \nu\|_{L^1(\area_{\mathcal D})} \rightarrow 0$ as $\epsilon\rightarrow 0$ for a bounded subset $\mathcal D$ by 
\eqref{CTEEstimate}. So for a smooth function $\varphi$ with compact support, we have
 \[
 \ \begin{aligned}
 \ &\int \bar \partial \varphi (z)\mathcal C(\eta)(z)\mathcal C(\nu)(z) d\area(z) \\
 \ = & \lim_{\epsilon\rightarrow 0 }\int \bar \partial \varphi (z)\tilde {\mathcal C}_\epsilon \nu (z)\mathcal C(\eta)(z) d\area(z) \\
\ = & \lim_{\epsilon\rightarrow 0 }\int \bar \partial (\varphi (z)\tilde {\mathcal C}_\epsilon \nu(z))\mathcal C(\eta)(z) d\area(z) - \lim_{\epsilon\rightarrow 0 }\int \varphi (z) \bar \partial ( \tilde {\mathcal C}_\epsilon \nu(z))\mathcal C(\eta)(z) d\area(z) \\
\ = & I - \lim_{\epsilon\rightarrow 0 }II_\epsilon.
\ \end{aligned}
 \]
By \eqref{CTEEstimate} and the assumption, we have
 \[
 \ |\tilde {\mathcal C}_\epsilon (\nu)(\lambda)|\le |\tilde {\mathcal C}_\epsilon (\nu) (\lambda) - \mathcal C_\epsilon (\nu) (\lambda)| + |\mathcal C_\epsilon (\nu) (\lambda)| \lesssim M,~\eta-a.a..
 \]
We find a sequence $\{\tilde {\mathcal C}_{\epsilon_k} \nu(z)\}$ converging to $F_2$ in $L^\infty (\eta)$ weak-star topology and   
 \[
 \ \tilde {\mathcal C}_{\epsilon_k} \nu(z) \rightarrow  \mathcal C \nu(z) = F_2(z), ~ \eta |_{\mathcal {ZD}(\nu)}-a.a.
 \]
by \eqref{CTEEstimate} and Corollary \ref{ZeroAC} since a zero $\gamma$ set is a zero $\eta$ set (see Lemma \ref{ZeroACEta}). By Lemma \ref{GammaExist}, we see that $\eta (\mathcal {ZD}(\eta) \cap \mathcal {ND}(\nu)) = 0.$ Therefore,
$\mathcal C \nu(z) = F_2(z), ~ \eta |_{\mathcal {ZD}(\eta)}-a.a.$
  Using the Lebesgue dominating convergence theorem, we get
 \[
 \ I = - \lim_{k\rightarrow \i}\int \varphi (z)\tilde {\mathcal C}_{\epsilon_k} \nu(z) \bar \partial\mathcal C(\eta)(z) d\area(z) = \pi \int \varphi (z)F_2(z) d\eta(z).
 \]
Now we estimate $II_\epsilon:$
 \begin{eqnarray}\label{CTDistributionEq1}
 \ \begin{aligned}\label{CTDistributionEq1}
\ II_\epsilon = &\int \varphi (z) \bar \partial ( \mathcal C(\phi_\epsilon* \nu)(z))\mathcal C\eta(z) d\area(z) \\
\  = & - \pi \int \varphi (z) \mathcal C\eta(z) (\phi_\epsilon* \nu)(z)d\area(z). \\
\  = & - \pi \int (\varphi \mathcal C\eta)* \phi_\epsilon(z)d\nu(z). \\
\ \end{aligned} 
 \end{eqnarray}
On the other hand,
\begin{eqnarray}\label{CTDistributionEq2}
 \ \begin{aligned}
 \ & | (\varphi \mathcal C\eta)* \phi_\epsilon(z) -  \varphi (z)( \mathcal C\eta)* \phi_\epsilon(z)|  \\
\ \lesssim &  \int |\varphi(w) - \varphi(z)| |\mathcal C\eta(w)|\phi_\epsilon(z-w)d\area(w) \\
\ \lesssim & \|\mathcal C\eta\|_\C \sup_{|w-z|\le \epsilon} |\varphi(w) - \varphi(z)| \rightarrow 0 \text{ as } \epsilon\rightarrow 0.
\ \end{aligned}
 \end{eqnarray}
 For $\lambda \in \C \setminus \mathcal X_\eta,$ $\lim_{\epsilon\rightarrow 0}\mathcal C_\epsilon (\eta)(\lambda) = \mathcal C (\eta)(\lambda)$ exists. 
By \eqref{CTEEstimate} and the assumption, $\tilde {\mathcal C}_\epsilon \eta (z)$ converges to $\mathcal C\eta (z)$ on $\mathcal {ZD}(\eta) \setminus \mathcal X_\eta$ and
$\|\tilde {\mathcal C}_\epsilon (\eta)\|_{L^\infty (|\nu|)} \lesssim 1.$
 We find a sequence $\{\tilde {\mathcal C}_{\epsilon_k} \eta(z)\}$ converging to $F_1$ in $L^\infty (|\nu|)$ weak-star topology and   
 \[
 \ \tilde {\mathcal C}_{\epsilon_k} \eta(z) \rightarrow  \mathcal C \eta(z) = F_1(z), ~ \nu|_{\mathcal {ZD}(\eta) \setminus \mathcal X_\eta}-a.a.
\]
(see \eqref{CTEEstimate}).
Combining with \eqref{CTDistributionEq1} and \eqref{CTDistributionEq2}, we conclude that
\[
\ \lim_{k\rightarrow \i}II_{\epsilon_k} = - \pi \lim_{k\rightarrow \i} \int \varphi (z) \tilde {\mathcal C}_{\epsilon_k} \eta(z)(z)d\nu(z) = - \pi \int \varphi (z) F_1(z)d\nu(z). 
\]
 The lemma is proved.
\end{proof}

The following lemma is a simple application of Theorem \ref{TolsaTheorem} (2).

\begin{lemma}\label{GLLemma} 
Let $X\subset \mathbb C$ and $a_1,a_2 \in \mathbb C.$ 
Let $\mathcal Q$ be any set in $\C$ with  $\gamma(\mathcal Q) = 0$.
Let $f_1(z)$ and $f_2(z)$ be functions on $\D(\lambda, \delta_0)\setminus \mathcal Q$ for some $\delta_0 > 0.$ 
 If $a_2\ne 0$ and
\[  
 \  \lim_{\delta \rightarrow 0} \dfrac{\gamma(\D(\lambda, \delta) \cap X \cap \{|f_i(z) - a_i| > \epsilon\})} {\delta}= 0, ~ i =1,2,
\]
for all $\epsilon > 0,$ then 
\[  
 \  \lim_{\delta \rightarrow 0} \dfrac{\gamma(\D(\lambda, \delta) \cap X \cap \{|\frac{f_1(z)}{f_2(z)} - \frac{a_1}{a_2}| > \epsilon\})} {\delta}= 0.
\]
 \end{lemma}

\begin{lemma}\label{RhoGContLemma}
If $f\in R^t(K,\mu),$ then there exists a subset $\mathcal Q_f$ with $\gamma(\mathcal Q_f) = 0$ such that $\rho(f)$ is $\gamma$-continuous at each point $\lambda\in \mathcal R \setminus 	\mathcal Q_f.$
\end{lemma}

\begin{proof}
There exists a subset $\mathcal Q_0$ with $\gamma(\mathcal Q_0)=0$ such that for $\lambda \in \mathcal R_0 \setminus \mathcal Q_0,$ there exists $j_0$ satisfying $\lim_{\epsilon \rightarrow 0} \mathcal C_\epsilon (g_{j_0}\mu) (\lambda) = \mathcal C(g_{j_0}\mu) (\lambda)$ exists, $\lim_{\epsilon \rightarrow 0} \mathcal C_\epsilon (fg_{j_0}\mu) (\lambda) = \mathcal C(fg_{j_0}\mu) (\lambda)$ exists, $\Theta_{g_{j_0}\mu}(\lambda) = \Theta_{fg_{j_0}\mu}(\lambda) = 0,$  $\mathcal C(g_{j_0}\mu) (\lambda) \ne 0,$ and $\rho(f)(\lambda) = \frac{\mathcal C(fg_{j_0}\mu)(\lambda)}{\mathcal C(g_{j_0}\mu)(\lambda)}$ by \eqref{RhoExistLemmaEq1}.
From Lemma \ref{CauchyTLemma}, we infer that $\mathcal C(g_{j_0}\mu)$ and  $\mathcal C(fg_{j_0}\mu)$ are $\gamma$-continuous at $\lambda.$ Hence, applying Lemma \ref{GLLemma} for $X = \C,$ we infer that $\frac{\mathcal C(fg_{j_0}\mu)(z)}{\mathcal C(g_{j_0}\mu)(z)}$ is $\gamma$-continuous at $\lambda.$ It follows from \eqref{RhoExistLemmaEq1} that $\rho(f)$ is $\gamma$-continuous at $\lambda.$

To deal with $\mathcal R_1,$ without loss of generality, we consider $\mathcal R_1 \cap \Gamma_1$ with the rotation angle $\beta_1 = 0.$  There exists a subset $\mathcal Q_1$ with $\gamma(\mathcal Q_1)=0$ such that for $\lambda \in \mathcal R_1 \cap \Gamma_1 \setminus \mathcal Q_1,$ there exist integers $j_0, j_1, j_2$ satisfying	$\lim_{\epsilon \rightarrow 0} \mathcal C_\epsilon (g_{j_k}\mu) (\lambda) = \mathcal C(g_{j_k}\mu) (\lambda)$ and $\lim_{\epsilon \rightarrow 0} \mathcal C_\epsilon (fg_{j_k}\mu) (\lambda) = \mathcal C(fg_{j_k}\mu) (\lambda)$ exist for $k=0,1,2,$ $\mathcal C(g_{j_0}\mu) (\lambda) \ne 0$ by Lemma \ref{NPMLemma}, $v^+(g_{j_1}\mu, \Gamma_1) (\lambda) \ne 0,$ $v^-(g_{j_2}\mu, \Gamma_1) (\lambda) \ne 0,$ and 
\[
\ \rho(f)(\lambda) = f(\lambda) = \dfrac{\mathcal C(fg_{j_0}\mu)(\lambda)}{\mathcal C(g_{j_0}\mu)(\lambda)}  = \dfrac{v^+(fg_{j_1}\mu, \Gamma_1)(\lambda)}{v^+(g_{j_1}\mu, \Gamma_1)(\lambda)} = \dfrac{v^-(fg_{j_2}\mu, \Gamma_1)(\lambda)}{v^-(g_{j_2}\mu, \Gamma_1)(\lambda)}
\]
by \eqref{RhoExistLemmaEq1} and \eqref{RhoExistLemmaEq2}. Set $X_1 = \Gamma_1,$ $X_2 = U_{\Gamma_1},$ and  $X_3 = L_{\Gamma_1}.$ Using \eqref{RhoExistLemmaEq1}, Theorem \ref{GPTheorem1}, and Lemma \ref{GLLemma}, we get
\[  
 \  \lim_{\delta \rightarrow 0} \dfrac{\gamma(\D(\lambda, \delta) \cap X_j \cap \{|\rho(f)(z) - \rho(f)(\lambda)| > \epsilon\})} {\delta}= 0,~ j=1,2,3.
\]
Applying Theorem \ref{TolsaTheorem} (2), we see that $\rho(f)$ is $\gamma$-continuous at $\lambda.$ 
\end{proof}

\begin{lemma}\label{zeroR}
For $f\in R^{t,\i}(K,\mu)$, if $\eta \in M_0^+(\C),$ $\eta$ is of $1$-linear growth, and $\|\mathcal C_\epsilon (\eta)\| \le 1$ for $\epsilon > 0$ such that 
$\mathcal C (\eta)(\lambda) = \rho(f)(\lambda), ~ \area_{\mathcal R}-a.a.$  and 
\begin{eqnarray}\label{zeroREq1}
\ \mathcal M_{g_j\mu}(z ),~ | \mathcal C_\epsilon (g_j\mu)(z)| \le M_j < \i,~\eta-a.a. \text{ for } j\ge 1,
\end{eqnarray}
then $\eta(\mathcal R) = 0.$
 \end{lemma}

\begin{proof}
Suppose that 
 $\gamma (\mathcal R \cap \mathcal {ND}(\eta)) > 0.$ 
Then using Lemma \ref{GammaExist}, we find a Lipschitz graph $\Gamma_0$ such that $\eta(\Gamma_0 \cap \mathcal R) > 0$. Without loss of generality, we assume that the rotation angle of $\Gamma_0$ is zero. Clearly, $\eta |_{\Gamma_0 \cap \mathcal R}$ is absolutely continuous with respect to $\mathcal H^1|_{\Gamma_0}.$ There exists a subset $\mathcal Q$ with $\gamma(\mathcal Q) = 0$ such that Theorem \ref{GPTheorem1} holds and $\rho(f)$ is $\gamma$-continuous at at $\lambda\in \Gamma_0 \cap \mathcal R\setminus \mathcal Q$ (Lemma \ref{RhoGContLemma}). Together with the assumption and \eqref{AreaGammaEq}, we get
\[
\ v^+(\eta, \Gamma_0)(\lambda) = v^-(\eta, \Gamma_0)(\lambda) = \rho(f)(\lambda)
\]
for $\lambda\in \Gamma_0 \cap \mathcal R\setminus \mathcal Q,$ which implies $\eta(\Gamma_0 \cap \mathcal R) = 0.$ This is a contradiction. Thus,
\[
 \  \Theta_\eta (\lambda) = 0, ~\mathcal C (\eta)(\lambda) = \rho(f)(\lambda), ~ \gamma|_{\mathcal R}-a.a..
 \]
Since a zero $\gamma$ set is also a zero $\eta$ set (see Lemma \ref{ZeroACEta}), we get
 \begin{eqnarray}\label{ENEstimateEq4}
 \ \eta(\mathcal R_1) = 0 \text{ and }\mathcal C (\eta)(\lambda) = \rho(f)(\lambda), ~ \eta |_{\mathcal R}-a.a..
 \end{eqnarray}
 Let $\mu = h\eta + \mu_s$ be the Radon Nikodym decomposition with respect to $\eta$, where $\mu_s\perp\eta$. Using \eqref{zeroREq1} and applying Lemma \ref{distributionLemma} for $\eta$ and $\nu = g_j\mu$, we have
 \[
 \ F_1g_j\mu + F_2 \eta = fg_j\mu.
 \]
Applying \eqref{RhoExistLemmaEq2} and \eqref{ENEstimateEq4}, we get
 \[
 \ F_1g_jh = \mathcal C (\eta)g_jh = \rho(f)g_jh = fg_jh, ~\eta|_{\mathcal R}-a.a.. 
 \]
Therefore, 
 \[
 \ F_2(z) = \mathcal C(g_j\mu)(z) = 0, ~\eta|_{\mathcal R}-a.a. 
 \]
for $j \ge 1$. Thus, $\eta (\mathcal R) = 0$ since $\eta (\mathcal R_1) = 0$ by \eqref{ENEstimateEq4} and $\mathcal R_0 \subset \mathcal N.$
\end{proof}

The following lemma generalizes \cite[Lemma 7.1]{y23}.

\begin{lemma}\label{ENEstimate}
For $\lambda\in \C$ and $\delta > 0,$ we have
 \begin{eqnarray}\label{ENEstimateWATEq}
 \ \lim_{N\rightarrow\infty} \gamma(\D(\lambda, \delta)\cap \mathcal E _N) \lesssim \gamma(\D(\lambda, 2\delta)\cap \mathcal F).
 \end{eqnarray}
\end{lemma} 

\begin{proof}
Set
 \[
 \ \epsilon_0 = \lim_{N\rightarrow\infty} \gamma(\D(\lambda, \delta)\cap \mathcal E_N) ( = \inf_{N\ge 1} \gamma(\D(\lambda, \delta)\cap \mathcal E_N)).
 \]
We assume $\epsilon_0 > 0$. From Lemma \ref{CTMaxFunctFinite}, there exists a Borel subset $F\subset \D(\lambda, \delta)$ such that $\gamma (\D(\lambda, \delta) \setminus F) < \frac{1}{2C_T}\epsilon_0$, where $C_T$ is the constant used in Theorem \ref{TolsaTheorem}; $\mathcal C_*(g_j\mu)(z) \le M_j < \infty$ for $z \in F$; and
$\mathcal M_{g_j\mu}(z ) \le M_j < \infty$ for $z \in F$. 

For $z \in \overline{F}$, let $\lambda_n\in F\cap \D(z, \frac 1n)$, we have 
 \[
 \ \dfrac{|g_j\mu|(\D(z, \delta))}{\delta}\le \dfrac{\delta+\frac 1n}{\delta}\mathcal M_{g_j\mu}(\lambda_n) \le \dfrac{\delta+\frac 1n}{\delta}M_j, 
 \]
which implies $\mathcal M_{g_j\mu}(z) \le M_j.$ 
 For $z \in F$, we have, by \eqref{CTEEstimate},
 \[
 \ |\tilde {\mathcal C}_\epsilon (g_j\mu)(z)|\le |\tilde {\mathcal C}_\epsilon (g_j\mu) (z) - \mathcal C_\epsilon (g_j\mu) (z)| + |\mathcal C_\epsilon (g_j\mu) (z)| \lesssim M_j.
 \]
 Because $\tilde {\mathcal C}_\epsilon(g_j\mu) (z)$ is continuous on $\overline{F}$, we have 
\begin{eqnarray}\label{ENEstimateEq1}
\ \mathcal M_{g_j\mu}(z ),~ | \mathcal C_\epsilon (g_j\mu)(z)|,~|\tilde {\mathcal C}_\epsilon (g_j\mu)(z)| \lesssim M_j\text{ for }z \in \overline{F}.
\end{eqnarray}

Using Theorem \ref{TolsaTheorem} (2), we have
 \[
 \ \begin{aligned}
 \ \gamma(F \cap \mathcal E_N) \ge & \dfrac{1}{C_T} \gamma(\D(\lambda, \delta)\cap \mathcal E_N) - \gamma (\D(\lambda, \delta) \setminus F) \\
\ \ge & \dfrac{1}{C_T} \gamma(\D(\lambda, \delta)\cap \mathcal E_N) - \dfrac{1}{2C_T} \epsilon_0 \\
\ \ge & \dfrac{1}{2C_T} \gamma(\D(\lambda, \delta)\cap \mathcal E_N).
 \ \end{aligned} 
 \]
 From Lemma \ref{BBFunctLemma}, we find $\eta_N \in M_0^+(F \cap \mathcal E_N)$ with $1$-linear growth, $ \| \mathcal C_\epsilon (\eta_N) \| _\C \le 1,$ and $\gamma(\D(\lambda, \delta)\cap \mathcal E_N) \lesssim \|\eta_N\|.$ 
 We may assume that $\eta_N \in C(\overline{F})^*\rightarrow \eta$ in weak-star topology. Clearly, $\text{spt}(\eta) \subset \overline{F}$, $\eta$ is $1$-linear growth, $\lim_{N\rightarrow \infty} \|\eta_N\| = \|\eta \|,$ and $\|\mathcal C_\epsilon (\eta)\| \le 1$ since $\mathcal C_\epsilon (\eta_N)$ converges to $\mathcal C_\epsilon (\eta)$ in $L^\infty (\mathbb C)$ weak-star topology. Hence,
 \begin{eqnarray}\label{ENEstimateEq20}
 \ \lim_{N\rightarrow\infty} \gamma(\D(\lambda,\delta)\cap \mathcal E_N) \lesssim \|\eta\|.	
 \end{eqnarray}

Using Lemma \ref{BBFunctLemma} (3), we conclude that there exists a sequence of $\{\epsilon_k\}$ such that $\mathcal C_{\epsilon_k}(\eta_N)$ converges to $f_N$ in $L^\infty(\mu)$ weak-star topology, $\|f_N\|_{L^\infty(\mu)} \le 1$,
 \begin{eqnarray}\label{ENEstimateEq2}
 \ \int f_Ng_jd\mu = - \int \mathcal C(g_j\mu) d\eta_N,
 \end{eqnarray}
and
 \begin{eqnarray}\label{ENEstimateEq3}
 \ \int \dfrac{f_N(z) - f_N(\lambda)}{z - \lambda}g_j(z)d\mu(z) = - \int \mathcal C(g_j\mu)(z) \dfrac{d\eta_N(z)}{z - \lambda}\text{ for }\lambda \in (\text{spt}\eta_N)^c.
 \end{eqnarray}
We may assume that $f_N$ converges to $f$ in $L^\infty(\mu)$ weak-star topology. By \eqref{ENEstimateEq2},
 \[
 \ \left | \int fg_jd\mu \right |= \lim_{N\rightarrow \infty}\left | \int f_Ng_jd\mu \right |\le \lim_{N\rightarrow \infty} \int |\mathcal C(g_j\mu)| d\eta_N \le \lim_{N\rightarrow \infty} \dfrac{\|\eta_N\|}{N} = 0,
 \]
which implies $f\in R^{t,\i}(K, \mu)$. By passing to a subsequence, we may assume that $\mathcal C(\eta_N)$ converges to $H(z)$ in $L^\infty (\mathbb C)$ weak-star topology. Let $\varphi$ be a smooth function with compact support. Then $\CT(\varphi\area)$ is continuous and we have
\[
\begin{aligned}
\ \int \varphi H d\area = & \lim_{N\rightarrow\i} \int \varphi \CT \eta_N d\area = - \lim_{N\rightarrow\i} \int \CT (\varphi\area) d\eta_N  \\
\ = &- \int \CT (\varphi\area) d\eta = \int \varphi \CT \eta d\area.
\end{aligned}
\]
Hence, $H = \CT \eta.$ Let $n < N$ and let $\phi$ be a bounded Borel function supported in $\C \setminus \mathcal E_n.$ Since $\mathcal E_N\subset \mathcal E_n, $  by \eqref{ENEstimateEq3}, we get
\[
\begin{aligned}
\ & \left |\int (\CT (f_Ng_j\mu)(\lambda) - \CT (\eta_N)(\lambda)\CT (g_j\mu)(\lambda) )\phi(\lambda)d\area \right | \\
\ = & \left | \int \CT (\phi\area)(z)\mathcal C(g_j\mu)(z) d\eta_N(z)\right | \\
\  \le & \dfrac{\|\CT (\phi\area)\|\|\eta_N\|}{N} \rightarrow 0,\text{ as }N\rightarrow\i.
\end{aligned}
\]
Clearly, $\int \CT(f_Ng_j\mu)(\lambda)\phi(\lambda)d\area \rightarrow \int \CT(fg_j\mu)(\lambda)\phi(\lambda)d\area$ as $N \rightarrow \i.$ Therefore,
\[
\ \int \CT (fg_j\mu)(\lambda) \phi(\lambda)d\area = \int \CT (\eta)(\lambda)\CT (g_j\mu)(\lambda) )\phi(\lambda)d\area,
\] 
which implies
 \[
 \ \CT (fg_j\mu)(\lambda) = \CT (\eta)(\lambda)\CT (g_j\mu)(\lambda),~ \area |_{\C \setminus \mathcal E_n}-a.a..
 \]
 As $\mathcal F \approx \cap_{n=1}^\i \mathcal E_n,~\area-a.a.,$ by \eqref{acZero}, we infer that 
 \[
 \ \CT (fg_j\mu)(\lambda) = \CT (\eta)(\lambda)\CT (g_j\mu)(\lambda),~ \area-a.a.,
 \]
 which implies $\mathcal C (\eta)(\lambda) = \rho(f)(\lambda), ~ \area_{\mathcal R}-a.a..$
Thus, by  \eqref{ENEstimateEq1} and Lemma \ref{zeroR}, we get $\eta(\mathcal R)=0.$
There is an open subset $O$ such that $\text{spt}(\eta)\cap \mathcal R \subset O$ and $\eta(O) \le \frac 14 \|\eta\|$. Using Proposition \ref{GammaPlusThm} (2), there exists a subset $A$ such that $\|\eta\| \le 2\|\eta|_A\|$ and $N_2(\eta|_A) \lesssim 1.$ Then $\|\eta\| \le 4\|\eta|_{A\setminus O}\|$ and $N_2(\eta|_{A\setminus O}) \lesssim 1.$
Hence, $\text{spt}(\eta |_{A\setminus O} ) \subset \text{spt}(\eta)\cap\mathcal F$ and by Proposition \ref{GammaPlusThm} (3) and Theorem \ref{TolsaTheorem} (1), we get
 \[
 \ \|\eta\| \lesssim \gamma (\text{spt}(\eta |_{A\setminus O}) \lesssim \gamma(\text{spt}(\eta)\cap\mathcal F).
 \]
The proof now follows from \eqref{ENEstimateEq20}.
\end{proof}

The proof of the following lemma is the same as that of \cite[Lemma 7.2]{y23} if \cite[Lemma 7.1]{y23} is replaced by Lemma \ref{ENEstimate}.

\begin{lemma} \label{RTIntegralLemma}
Let $F\in L^\i (\area_{\mathcal R})$ and $\|F\|_{L^\i (\area_{\mathcal R})} \le 1.$ Suppose that for $\epsilon > 0,$ there exists $A_\epsilon \subset \mathcal R$ with $\gamma(A_\epsilon) < \epsilon$ and there exists $F_{\epsilon, N}\in R(K_{\epsilon, N})$ (uniform closure of $\text{Rat}(K_{\epsilon, N})$ in $C(K_{\epsilon, N})$), where $\mathcal R _{\epsilon, N} =  \mathcal R  \setminus (A_\epsilon\cup \mathcal E_N)$ and $K_{\epsilon, N} = \overline{\mathcal R _{\epsilon, N}},$  such that $\|F_{\epsilon, N}\| \le 2$ and 
\begin{eqnarray}\label{MLemma1RhoF}
 \ F (z) = F_{\epsilon, N}(z), ~ \area_{\mathcal R _{\epsilon, N}}-a.a..
 \end{eqnarray}
Then for  $\varphi$ a smooth function with support in $\D(\lambda, \delta)$,
 \[
 \ \left |\int F(z) \bar \partial \varphi (z) d\area_{\mathcal R} (z) \right | \lesssim \delta \|\bar \partial \varphi\| \gamma(\D(\lambda, 2\delta) \cap \mathcal F).
 \]
\end{lemma}

Now we are ready to prove Lemma \ref{MLemma1} as the following.

\begin{proof} (Lemma \ref{MLemma1}): Let $f\in R^{t,\i}(K,\mu)$ and $\{r_n\}\subset \text{Rat}(K)$ such that $\|r_n - f\|_{L^t(\mu)}\rightarrow 0$ and $r_n(z)\rightarrow  f(z),~\mu-a.a.$ as $n\rightarrow \i.$
Using Lemma \ref{ConvergeLemma} (1), we find $A^1_\epsilon$ and a subsequence $\{r_{n,1}\}$ of $\{r_{n}\}$ such that $\gamma(A^1_\epsilon) < \frac{\epsilon}{2C_T}$ and $\{\mathcal C(r_{n,1}g_1\mu)\}$ uniformly converges to $\mathcal C(fg_1\mu)$ on $\C \setminus A^1_\epsilon$. Then we find $A^2_\epsilon$ and a subsequence $\{r_{n,2}\}$ of $\{r_{n,1}\}$ such that $\gamma(A^2_\epsilon) < \frac{\epsilon}{2^2C_T}$ and $\{\mathcal C(r_{n,2}g_2\mu)\}$ uniformly converges to $\mathcal C(fg_2\mu)$ on $\C \setminus A^2_\epsilon$. Therefore, we have a subsequence $\{r_{n,n}\}$ such that $\{\mathcal C(r_{n,n}g_j\mu)\}$ uniformly converges to $\mathcal C(fg_j\mu)$ on $\C \setminus A_\epsilon$ for all $j \ge 1$, where $A_\epsilon = \cup_j A^j_\epsilon$ and $\gamma(A_\epsilon) < \epsilon$ by Theorem \ref{TolsaTheorem} (2).  
From \eqref{RhoExistLemmaEq1}, we infer that $\{r_{n,n}\}$ uniformly tends to $F := \rho(f)$ on $\mathcal R _{\epsilon, N} := \mathcal R \setminus (A_\epsilon \cup \mathcal E_N).$ Thus, $\{r_{n,n}\}$ uniformly tends to $F_{\epsilon, N}\in R(K_{\epsilon, N})$ on $K_{\epsilon, N} := \overline{\mathcal R _{\epsilon, N}}.$ The proof now follows from Lemma \ref{RTIntegralLemma}.
\end{proof}

\section{\textbf{Decomposition Theorems for $\rtkm$}}

We do not need to assume that $S_\mu$ is pure in this section. In this case, there exists a partition $\{\Delta_{00}, \Delta_{01}\}$ of $\text{spt}\mu$ such that
\begin{eqnarray} \label{RTFirstDecompEq}
 \ R^t(K, \mu ) = L^t(\mu _{\Delta_{00}})\oplus R^t(K, \mu _{\Delta_{01}})
 \end{eqnarray}
 and $S_{\mu _{\Delta_{01}}}$ is pure. If $g\perp R^t(K, \mu ),$ then $g(z) = 0,~ \mu _{\Delta_{00}}-a.a..$ Hence, we see that $\mathcal F$ and $\mathcal R$ do not depend on the trivial summand $L^t(\mu _{\Delta_{00}}).$ Therefore, we will not distinguish $\mathcal F$ and $\mathcal R$ between $R^t(K, \mu)$ and $R^t(K, \mu _{\Delta_{01}}).$ We set $\mathcal F = \C$ and $\mathcal R = \emptyset$ if $\mu _{\Delta_{01}} = 0.$ For a Borel subset $\Delta$ with $\chi_\Delta\in \rtkm,$ let $\mathcal F_\Delta$ and $\mathcal R_\Delta$ denote the non-removable boundary and removable set for $R^t(K, \mu_\Delta),$ respectively. 

\begin{proposition} \label{RTFRProp1}
If $S_\mu$ on $\rtkm$ is pure, then the following properties hold:

(1) If $\Delta$ is a Borel subset and $\chi_\Delta\in \rtkm$, then $\rho(\chi_\Delta) = \chi_{\mathcal R_\Delta},~\gamma-a.a.$ and $R^t(K, \mu_\Delta) = R^t(\overline {\mathcal R_\Delta}, \mu_\Delta).$	

(2) Suppose that for $i=1,2,$ $\Delta_i$ is a Borel subset and $\chi_{\Delta_i}\in \rtkm.$ Then $\Delta_1 \cap \Delta_2 = \emptyset,~\mu-a.a.$ if and only if $\mathcal R_{\Delta_1} \cap \mathcal R_{\Delta_2} \approx \emptyset, ~\gamma-a.a..$	

(3) If $\{\Delta_i\}_{i = 1}^\infty$ is a Borel partition of $\text{spt}\mu$ such that $\chi_{\Delta_i}\in \rtkm,$ then
\[
\ \mathcal F \approx \bigcap_{i = 1}^\infty \mathcal F_{\Delta_i} \text{ and } \mathcal R \approx \bigcup_{i = 1}^\infty \mathcal R_{\Delta_i},~\gamma-a.a..
\]

(4) If $\Delta$ is a Borel subset and $\chi_\Delta\in \rtkm$ is a non-trivial characteristic function, then there exists a minimal $\chi_{\Delta_0}\in \rtkm$ such that $\Delta_0 \subset \Delta.$

(5) If $\mathcal R_0 \subset \mathcal R$ is a Borel subset and $\chi_{\mathcal R_0}\in H^\i (\mathcal R),$ then there exists a Borel subset $\Delta_0$ and $\mathcal Q \subset \mathcal R$ with $\gamma(\mathcal Q)=0$ such that $\mathcal R_0 \approx \mathcal R_{\Delta_0},~\area-a.a.$ and $\mathcal R_{\Delta_0}\setminus \mathcal Q \subset \Delta_0  \subset  \overline {\mathcal R_{\Delta_0}},~\mu-a.a..$ In particular, if $U$ is an open subset and $U\subset \mathcal R_{\Delta_0},~\gamma-a.a.,$ then $U \subset  \Delta_0,~\mu-a.a..$
\end{proposition}

\begin{proof} (1): $\rho(\chi_\Delta) = \chi_{\mathcal R_\Delta},~\gamma-a.a.$ follows from \eqref{RhoExistLemmaEq1}. Using Proposition \ref{NF0Prop} (2), we get $R^t(K, \mu_\Delta) = R^t(\overline {\mathcal R_\Delta}, \mu_\Delta).$	

(2) is trivial since from (1), we have
\[
\ \rho(\chi_{\Delta_1 \cap \Delta_2}) = \rho(\chi_{\Delta_1 })\rho(\chi_{ \Delta_2}) = \chi_{\mathcal R_{\Delta_1}} \chi_{\mathcal R_{\Delta_2}} = \chi_{\mathcal R_{\Delta_1} \cap \mathcal R_{\Delta_2}}.
\]

(3): $\Lambda_i = \{\chi_{\Delta_i}g_j\} \subset R^t(K, \mu _{\Delta_i})^\perp$ is also a dense subset. It is clear that
\[
\ \begin{aligned}
 \ \CT (g_j\mu)(z) \approx & \sum_{i=0}^\i \CT (\chi_{\Delta_i} g_j\mu)(z),~\gamma-a.a.; \\
 \ v^+ (g_j\mu, \Gamma_n, \beta_n)(z) \approx & \sum_{i=0}^\i v^+ (\chi_{\Delta_i} g_j\mu, \Gamma_n, \beta_n)(z),~\gamma-a.a.; \\
 \ v^- (g_j\mu, \Gamma_n, \beta_n)(z) \approx & \sum_{i=0}^\i v^- (\chi_{\Delta_i} g_j\mu, \Gamma_n, \beta_n)(z),~\gamma-a.a.; \\
 \ \mathcal {ZD} (g_j\mu) \approx & \bigcap_{i=0}^\i \mathcal {ZD} (\chi_{\Delta_i} g_j\mu),~\gamma-a.a.; \\
 \ \mathcal {ND} (g_j\mu) \approx & \bigcup_{i=0}^\i \mathcal {ND} (\chi_{\Delta_i} g_j\mu),~\gamma-a.a..
 \ \end{aligned}	
\]
With above equations, it is straightforward to prove (3).

(4): Clearly, $\area(\mathcal R_\Delta) > 0$. There exists $g\perp R^t(K, \mu_\Delta)$ and $\lambda_0 \in \mathcal R_\Delta$ such that 
 \[
 \ \int \dfrac{1}{|z-\lambda_0|}|g(z)|d\mu  < \infty
 \]
and $\mathcal C(g\mu)(\lambda_0) \ne 0$. Define
\[
\ e_{\lambda_0}(f) := \dfrac{\mathcal C(fg\mu)(\lambda_0)}{\mathcal C(g\mu)(\lambda_0)} = \rho(f)(\lambda_0),~f\in R^{t, \infty}(K, \mu_\Delta).
\] 	 
It is easy to verify $e_{\lambda_0}$ a weak-star continuous multiplicative linear functional on $R^{t, \infty}(K, \mu_\Delta)$ such that $e_{\lambda_0}(f) = f(\lambda_0)$ for each $f \in \text{Rat}(K).$

Suppose that $R^{t, \infty}(K, \mu_\Delta)$ does not contain a non-trivial minimal characteristic function. Set
 \[
 \ \mathcal B = \{B\subset \Delta:~ \chi_B \in R^{t, \infty}(K, \mu_\Delta),~e_{\lambda_0}(\chi_B) = 1\}.
 \]
Then $\chi_\Delta\in \mathcal B \ne \emptyset$. For $B_1, B_2\in \mathcal B$,  $e_{\lambda_0}(\chi_{B_1\cap B_2}) = e_{\lambda_0}(\chi_{B_1})e_{\lambda_0}(\chi_{B_2})(\lambda_0) = 1$. We find $\Delta\supset B_n\supset B_{n+1}$ such that 
 \[
 \ \mu(B_n) \rightarrow b = \inf_{B\in \mathcal B}\mu(B).
 \]
Clearly $B =\cap B_n\in \mathcal B$ and $ b = \mu(B) > 0$.  From the assumption, we get that $\chi_B$ is not a minimal characteristic function. Hence, there exists $B_0\subset B$ with $\mu(B_0) > 0$ and $\mu(B\setminus B_0) > 0$ such that $\chi_{B_0}, ~\chi_{B\setminus B_0}  \in R^{t, \infty}(K, \mu_\Delta)$. Since 
 \[
 \ 1 = e_{\lambda_0}(\chi_{B}) = e_{\lambda_0}(\chi_{B_0}) + e_{\lambda_0}(\chi_{B\setminus B_0})
 \]
 and 
 \[ 
 \ e_{\lambda_0}(\chi_{B_0}) e_{\lambda_0}(\chi_{B\setminus B_0}) = e_{\lambda_0}(\chi_{B_0}\chi_{B\setminus B_0}) = 0,
 \]
we see that $B\setminus B_0\in \mathcal B$ or $B_0\in \mathcal B$, which contradicts the definition of $b$. Therefore, there exists a non-trivial minimal characteristic function in $R^{t, \infty}(K, \mu_\Delta).$

(5): Let $f_0 = \chi_{\mathcal R_0}.$ By Theorem \ref{mainThmR}, if $\tilde f_0 := \rho^{-1}(f_0),$ then 
\[
\ \tilde f_0^2 = \rho^{-1}(f_0)\rho^{-1}(f_0) = \rho^{-1}(f_0^2) = \tilde f_0.
\]
Hence, there exists a Borel subset $\Delta_0$ such that $\tilde f_0 = \chi_{\Delta_0}.$ From (1) and \eqref{RhoExistLemmaEq2}, we get $\mathcal R_{\Delta_0}\setminus \mathcal Q \subset \Delta_0  \subset  \overline {\mathcal R_{\Delta_0}},~\mu-a.a..$ If an open subset $U \subset \mathcal R_{\Delta_0}, ~\gamma-a.a.,$ then $\CT(\tilde f_0g\mu)(z) = \CT(g\mu)(z),~\area_U-a.a.$ for $g\perp \rtkmu.$ By \eqref{CTDistributionEq}, we have $U \subset  \Delta_0,~\mu-a.a.$ since $S_\mu$ is pure.
\end{proof}

\begin{theorem}\label{DecompTheoremRT} Let $K$ be a compact subset, $1 \le t < \i,$ and $\mu \in M_0^+(K).$ 
Then there exists a Borel partition $\{\Delta_i\}_{i\ge 0}$ of $\text{spt}(\mu )$ and compact subsets $\{K_i\}_{i=1}^\infty$ such that $\Delta_i \subset K_i$ for $i \ge 1$,
 \[
 \ R^t(K,\mu) = L^t(\mu |_{\Delta_0})\oplus \bigoplus_{i=1}^\infty R^t(K_i, \mu_{\Delta_i})
 \]
 and the following statements are true: 

(1) If $i \ge 1$, then $R^t(K_i, \mu_{\Delta_i})$ contains no non-trivial characteristic functions. 

(2) If $i \ge 1$, then $K_i = \overline{\mathcal R_{\Delta_i}}$ and $R^t(K, \mu_{\Delta_i}) = R^t(K_i, \mu_{\Delta_i}).$

(3) If $i \ge 1$, then the map $\rho_i$ is an isometric isomorphism and a weak$^*$ homeomorphism from $R^{t,\i}(K_i, \mu_{\Delta_i})$ onto $H^\infty(\mathcal R_{\Delta_i})$. 
\end{theorem}

\begin{proof}
Using \cite[Theorem 1.6 on page 279]{conway}, we find disjoint Borel subsets $\{\Delta_i\}_{i\ge 1}$ such that $\chi_{\Delta_i} \in R^{t,\i}(K_i, \mu)$ is a minimal characteristic function.	 Set $\Delta_0 = K \setminus \cup_{i=1}^\i \Delta_i.$ Then
\[
 \ R^t(K,\mu) = R^t(K, \mu_{\Delta_0})\oplus \bigoplus_{i=1}^\infty R^t(K, \mu_{\Delta_i})
 \]
 and $R^t(K, \mu_{\Delta_0})$ has no minimal characteristic functions. Applying Proposition \ref{RTFRProp1} (4), we conclude that $\mathcal R_{\Delta_0} = \emptyset,$ which implies $R^t(K, \mu_{\Delta_0}) = L^t(\mu_{\Delta_0}).$
 
 (2) follows from Proposition \ref{RTFRProp1} (1). (1) is trivial. Theorem \ref{mainThmR} implies (3).
\end{proof}

 A point $z_0\in K$ is called a \textit{bounded point evaluation} for $R^t(K, \mu)$ if $r\mapsto r(z_0)$ defines a bounded linear functional for functions in $\mbox{Rat}(K)$ with respect to the $L^t(\mu)$ norm. The collection of all such points is denoted
 $\mbox{bpe}(R^t(K, \mu)).$  If $z_0$ is in the interior of $\mbox{bpe}(R^t(K, \mu))$ 
and there exist positive constants $\delta$ and $M$ such that $|r(z)| \leq M\|r\|_{L^t(\mu)}$, whenever $z\in \D(z_0, \delta)$ 
and $r\in \mbox{Rat}(K),$ then we say that $z_0$ is an 
\textit{analytic bounded point evaluation} for $R^t(K, \mu).$ The collection of all such 
points is denoted $\mbox{abpe}(R^t(K, \mu)).$ 

In 1991, J. Thomson \cite{thomson} obtained a celebrated decomposition theorem for $P^t(\mu),$ the closed subspace of $L^t(\mu)$ spanned by the analytic polynomials. 
J. Conway and N. Elias studied the set analytic bounded point evaluations for certain $\rtkmu.$ Later, J. Brennan \cite{b08} generalized Thomson's theorem to $R^t(K, \mu)$ when the diameters of the components of $\mathbb C\setminus K$ are bounded below. In all above cases, we will see below that $\mathcal R$ equals the set of analytic bounded point evaluations. However, it may happen that $R^t(K,\mu) \ne L^t(\mu),$ $\mbox{abpe}(R^t(K,\mu)) = \emptyset,$ and $\mathcal R \ne \emptyset.$ Examples of this phenomenon can be constructed, where $K$ is a Swiss cheese set (with empty interior, see \cite{b71} and \cite{f76}).

 \begin{lemma}\label{ABPEEqLemma}
 Let $\rtkmu$ be decomposed as in \eqref{RTFirstDecompEq}. Then $\text{abpe}(R^t(K, \mu _{\Delta_{01}})) = \text{abpe}(R^t(K, \mu)).$	
 \end{lemma}

\begin{proof}
  Let $\D(\lambda_0, \delta) \subset \text{abpe}(R^t(K, \mu ))$ such that for $\lambda\in \D(\lambda_0, \delta)$ and $r\in \text{Rat}(K),$ $|r(\lambda)| \le M \|r\|_{L^t(\mu)}$ for some $M > 0$ and $r(\lambda) = (r,k_\lambda),$  where $k_\lambda \in L^s(\mu)$ and $\|k_\lambda\|_{L^s(\mu)}\le M.$ Because $(z-\lambda)\bar k_\lambda\perp R^t(K, \mu ),$ we get $k_\lambda(z) = 0,~ \mu _{\Delta_{00}}-a.a.$ if $\mu(\{\lambda\}) = 0.$ Let $\{\lambda_n\}$ be the set of atoms for $\mu.$ Then $|r(\lambda)| \le M \|r\|_{L^t(\mu _{\Delta_{01}})}$ for $\lambda\in \D(\lambda_0, \delta) \setminus \{\lambda_n\}.$ Now for $|\lambda - \lambda_0| < \frac{\delta}{2},$ we have
  \[
  \ |r(\lambda)| \lesssim \dfrac{1}{\pi\delta^2} \int_{\D(\lambda_0, \delta) \setminus \{\lambda_n\}} |r(z)| d\area \lesssim M \|r\|_{L^t(\mu _{\Delta_{01}})}.
  \]
  Thus, $\lambda_0 \in  \text{abpe}(R^t(K, \mu _{\Delta_{01}})).$ The lemma is proved.
  \end{proof}

From Lemma \ref{ABPEEqLemma}, we see that $\text{abpe}(R^t(K, \mu))$ does not depend on the trivial summand $L^t(\mu _{\Delta_{00}}).$

We let $\partial_e K$ (the exterior boundary of $K$) denote the union of
   the boundaries of all the components of $\C \setminus K.$ Define
\begin{eqnarray}\label{BOne}
 \ \partial_1 K = \left \{\lambda \in K:~ \underset{\delta\rightarrow 0}{\overline\lim} \dfrac{\gamma(\D(\lambda,\delta)\setminus K)}{\delta} > 0 \right \}.
 \end{eqnarray}
Obviously, $\partial_e K \subset \partial_1 K \subset \partial K.$ If the diameters of  the components of $\C\setminus K$ are bounded below, then there exist $\epsilon_0 > 0$ and $\delta_0 > 0$ such that for each $\lambda \in \partial K,$
\begin{eqnarray}\label{CDiamBBEqn}
\ \gamma (\D(\lambda,\delta)\setminus K) \ge \epsilon_0 \delta \text{ for }\delta < \delta_0.
\end{eqnarray}
Clearly, if $K$ satisfies \eqref{CDiamBBEqn}, then $\partial K = \partial_1 K.$ Conversely,  it is straightforward to construct a compact subset $K$ such that $\partial K = \partial_1 K$ and $K$ does not satisfy \eqref{CDiamBBEqn}.

\begin{proposition} \label{RTFRProp2} 
$\partial_1 K \subset \mathcal F, ~\gamma-a.a.$. Consequently, if $K$ satisfies \eqref{CDiamBBEqn}, then $\partial K \subset \mathcal F, ~\gamma-a.a.$.
\end{proposition}

\begin{proof}
The proposition follows from Theorem \ref{FCharacterization} and the fact that for $\lambda \in \partial_1 K$, 
 \[
 \ \D(\lambda, \delta) \setminus K \subset \D(\lambda, \delta) \cap\mathcal E_N.
 \]
 \end{proof}

 The following Lemma is from Lemma B in \cite{ars09}.

\begin{lemma} \label{lemmaARS}
There are absolute constants $\epsilon _1, C_1 > 0$ with the
following property. If $R > 0$ and $E \subset  \overline{\D(0, R)}$ with 
$\gamma(E) < R\epsilon_1$, then
\[
\ |p(\lambda)| \le \dfrac{C_1}{\pi R^2} \int _{\overline{\D(0, R)}\setminus E} |p|\, d\area
\]
for all $\lambda$ in $\D(0, \frac{R}{2})$ and all analytic polynomials $p$.
\end{lemma}

The theorem below provides  an important relation between $\text{abpe}(R^t(K,\mu))$ and $\mathcal R.$ 

\begin{theorem}\label{ABPETheoremRT} 
The following property holds:
 \begin{eqnarray}\label{ABPETheoremRTEq1}
 \ \text{abpe}(R^t(K,\mu)) \approx \text{int}(K) \cap \mathcal R,~ \gamma-a.a..
 \end{eqnarray}
More precisely, the following statements are true:

(1) If $\lambda_0 \in \text{int}(K)$ and there exists $N\ge 1$ such that 
 \begin{eqnarray}\label{ABPETheoremRTEq2}
 \ \lim_{\delta \rightarrow 0} \dfrac{ \gamma (\mathcal E_N \cap \D(\lambda_0, \delta))}{\delta} = 0,
 \end{eqnarray}
then $\lambda_0\in \text{abpe}(R^t(K,\mu))$.

(2)
\[
 \  \text{abpe}(R^t(K,\mu)) \subset \text{int}(K) \cap \mathcal R,~ \gamma-a.a..
 \]
\end{theorem}

\begin{proof} From Lemma \ref{ABPEEqLemma}, we may assume that $S_\mu$ is pure.

(1): If $\lambda_0 \in \text{int}(K)$  satisfies \eqref{ABPETheoremRTEq2}, then we  choose $\delta > 0$ small enough such that $\D(\lambda_0, \delta) \subset int(K)$ and 
$\gamma (E:= \mathcal E_N \cap \D(\lambda_0, \delta)))\le \epsilon_1 \delta$, where $\epsilon_1$ is from Lemma \ref{lemmaARS}. Hence, using Lemma \ref{lemmaARS}, we conclude
 \[
 \ \begin{aligned}
 \ |r(\lambda )| \lesssim & \dfrac{1}{\pi \delta^2} \int _{\D(\lambda_0, \delta) \setminus E} |r(z)| d\area(z) \\
 \ \lesssim & \dfrac{N}{\pi \delta^2} \int _{\D(\lambda_0, \delta)} |r(z)| \max_{1\le j \le N} |\mathcal C(g_j\mu)(z)| d\area(z) \\
\ \lesssim & \dfrac{N}{\pi \delta^2} \int _{\D(\lambda_0, \delta)} \max_{1\le j \le N} |\mathcal C(rg_j\mu)(z)| d\area(z) \\
\ \lesssim & \dfrac{N}{\pi \delta^2} \sum_{j = 1}^N\int _{\D(\lambda_0, \delta)}  |\mathcal C(rg_j\mu)(z)| d\area(z) \\
\ \lesssim & \dfrac{N}{\pi \delta^2} \sum_{j = 1}^N\int \int _{\D(\lambda_0, \delta)} \left |\dfrac{1}{z-w} \right | d\area(z) |r(w)||g_j (w)| d\mu (w) \\
\ \lesssim & \dfrac{N}{\delta} \sum_{j = 1}^N \|g_j\|_{L^{s}(\mu)} \|r\|_{L^{t}(\mu)}
 \ \end{aligned}
 \]
for all $\lambda$ in $\D(\lambda_0, \frac{\delta}{2})$ and all $r\in Rat(K)$. This implies that $\lambda_0\in abpe(R^t(K,\mu))$.

(2): Let $E \subset G\cap \mathcal F$ be a compact subset with $\gamma(E) > 0$, where $G$ is a connected component of $abpe(R^t(K,\mu))$. By Lemma \ref{BBFRLambda}, there exists $f\in R^{t,\i}(K,\mu)$ that is bounded and analytic on $E^c$ such that  
 \[
 \ \|f\|_{L^\infty(\mu)} \lesssim 1,~ f(\infty) = 0,~ f'(\infty) = \gamma(E),
 \]
and 
\[
\ \dfrac{f(z) - f(\lambda)}{z - \lambda} \in R^{t,\i}(K,\mu),~ \lambda \in E^c.
\]
Let $r_n\in Rat(K)$ such that $\|r_n-f\|_{L^t(\mu)} \rightarrow 0$. Hence, $r_n$ uniformly tends to an analytic function $f_0$ on compact subsets of $G$ and $\frac{f-f_0(\lambda)}{z-\lambda} \in R^t(K,\mu)$ for $\lambda\in G$. Therefore, $f(z) = f_0(z)$ for $z\in G\setminus E.$  Thus, the function $f_0(z)$ can be analytically extended to $\mathbb C_\i$ and $f_0(\infty) = 0$. So $f_0=0.$ This contradicts $f'(\infty) \ne 0$. 

\eqref{ABPETheoremRTEq1} now follows from Theorem \ref{FCharacterization}.
\end{proof}

From Proposition \ref{RTFRProp2} and Theorem \ref{ABPETheoremRT}, under the assumptions of \cite{thomson}, \cite{ce93}, or \cite{b08}, we conclude that $\mathcal R \approx \text{abpe}(\rtkmu),~\gamma-a.a..$ 

\begin{theorem}\label{thmA}
 Let $K\subset \C$ be a compact subset, $1\le t < \infty,$ and $\mu\in M_0^+(K).$ Let $\text{abpe}(\rtkmu) = \cup_{n=1}^\i U_n,$ where $U_n$ is a connected component.   
Then the following statements are equivalent.
\newline
(1) $\partial U_i \cap \partial  K \subset \mathcal F, ~\gamma-a.a.$ for all $i \ge 1.$
\newline
(2) There is a Borel partition $\{\Delta_i\}_{i=0}^\infty$ of $\mbox{spt}(\mu)$ and compact subsets $\{K_i\}_{i=0}^\infty$ such that $\Delta_i \subset K_i$ for $i \ge 0$,
 \[
 \ R^t(K, \mu ) = R^t(K_0,\mu _{\Delta_0})\oplus \bigoplus _{i = 1}^\infty R^t(K_i, \mu _{\Delta_i}),
 \]
and the following properties hold:
\begin{itemize}	

\item[(a)] $K_0$ is the spectrum of $S_{\mu _{\Delta_0}}$ and $\text{abpe}(R^t(K_0,\mu _{\Delta_0})) = \emptyset.$

\item[(b)] If $i \ge 1$, then $S_{\mu _{\Delta_i}}$ on $R^t(K_i, \mu _{\Delta_i})$ is irreducible, that is, $R^t(K_i, \mu _{\Delta_i})$ contains no non-trivial characteristic functions. 

\item[(c)] If  $i \ge 1,$ then $U_i =\mbox{abpe}( R^t(K_i, \mu _{\Delta_i}))$ and $K_i = \overline{U_i}$.

\item[(d)] If  $i \ge 1$, then the evaluation map $\rho_i: f\rightarrow f|_{U_i}$ is an isometric isomorphism and a weak-star homeomorphism from $R^{t, \i}(K_i, \mu _{\Delta_i})$ onto $H^\infty(U_i)$.
\end{itemize}
\end{theorem}

\begin{proof} 
(Theorem \ref{thmA} (1)$\Rightarrow$(2)):
From \eqref{RTFirstDecompEq} and Lemma \ref{ABPEEqLemma}, we may assume that $S_\mu$ is pure.
From \eqref{ABPETheoremRTEq1}, we see that $\partial U_i \cap \text{int}(K) \subset \mathcal F,~\gamma-a.a..$ Hence, the assumption (1) implies $\partial U_i \subset \mathcal F,~\gamma-a.a..$ Thus, $\chi_{U_i}\in H^\i(\mathcal R).$
Using Proposition \ref{RTFRProp1} (5), we infer that there exists a Borel subset $\Delta_i$  such that $\chi_{\Delta_i} \in R^t(K, \mu),$ $U_i \approx \mathcal R_{\Delta_i},~\area-a.a.,$ and $U_i \subset \Delta_i \subset \overline U_i, ~\mu-a.a..$ From Proposition \ref{RTFRProp1} (1)\&(2), we see that
\[
\ \rho(\chi_{\Delta_i}\chi_{\Delta_m}) = \chi_{\mathcal R_{\Delta_i}}\chi_{\mathcal R_{\Delta_m}} = 0 \text{ for } i \ne m,
\]
which implies $\Delta_i \cap \Delta_m = \emptyset,~\mu-a.a..$ Set $\Delta_0 = K\setminus \cup_{i=1}^\i \Delta_i.$ Then $\{\Delta_i\}_{i \ge 0}$ is a Borel partition of $\text{spt}\mu$ and
\[
 \ R^t(K, \mu ) = R^t(K, \mu _{\Delta_0})\oplus \bigoplus _{i = 1}^\infty R^t(K, \mu _{\Delta_i}).
 \]

(c):  Set $K_i = \overline{U_i}.$ Then $K_i = \overline{\mathcal R_{\Delta_i}}$ and by Proposition \ref{RTFRProp1} (1), $R^t(K, \mu_{\Delta_i}) = R^t(K_i, \mu_{\Delta_i}).$ Clearly, $\partial U_i \subset \mathcal F \subset \mathcal F_{\Delta_i},~\gamma-a.a.,$ so $\mathcal R_{\Delta_i} \subset U_i \subset \text{int}(K_i),~\gamma-a.a..$ Therefore, by Theorem \ref{ABPETheoremRT}, we get $\mathcal R_{\Delta_i} \approx  \mbox{abpe}( R^t(K_i, \mu _{\Delta_i})),~\gamma-a.a..$ Hence, $U_i =  \mbox{abpe}( R^t(K_i, \mu _{\Delta_i})).$ This proves (c).

(d) follows from Theorem \ref{mainThmR}.

(b) follows from (d) since $U_i$ is connected and $H^\infty (U_i)$ contains no non-trivial characteristic functions.

(a) follows from Proposition \ref{RTFRProp1} (1)\&(2) and Theorem \ref{ABPETheoremRT}.
 \end{proof}

To prove Theorem \ref{thmA} (2)$\Rightarrow$(1), we need the following lemma.

\begin{lemma}\label{IForHILemma}
Let $U$ be a bounded open connected subset satisfying $K \subset \overline U.$ Suppose that $S_\mu$ on $R^t(K,\mu)$ is pure and $\mathcal I$ is an algebraic homomorphism from $H^\i (U)$ to $R^{t, \i}(K,\mu)$ that sends $1$ to $1$ and $z$ to $z.$
Then $\partial U \subset \mathcal F.~ \gamma-a.a..$
\end{lemma}

\begin{proof}
Let $W\supset U$ be an open subset.
For $f\in H^\i (W)$ and $\lambda\in W,$ we have $f_\lambda (z) := \frac{f(z) - f(\lambda)}{z - \lambda}\in H^\i (W).$ Hence,
\[
\ \mathcal I(f)(z) - f(\lambda) = (z - \lambda) \mathcal I(f_\lambda)(z),
\]
which implies $\frac{\mathcal I(f)(z) - f(\lambda)}{z - \lambda} \in R^{t, \i}(K,\mu).$ Thus, by Corollary \ref{ZeroAC}, we get
\begin{eqnarray} \label{IForHILemmaEq0}
\ \CT(\mathcal I (f) g\mu)(\lambda) = f(\lambda)\CT(g\mu)(\lambda),~\gamma|_W-a.a.,\text{ for } g\perp \rtkmu. 
\end{eqnarray}
Therefore, 
\begin{eqnarray} \label{IForHILemmaEq1}
\ \rho (\mathcal I (f))(\lambda) = f(\lambda),~\gamma|_{W\cap \mathcal R}-a.a.. 
\end{eqnarray}

Claim: If $\Gamma$ is a (rotated) Lipschitz graph, then $\gamma(\partial U \cap \Gamma \cap \mathcal R) = 0.$

Without loss of generality, we assume the rotation angle of $\Gamma$ is zero. Suppose there exists a compact subset $E \subset \partial U \cap \Gamma \cap \mathcal R$ with $\gamma(E) > 0.$
Let $f$ be an analytic function on $W:=\C_\i\setminus E$ with $\|f\| = 1,$ $f(\i) = 0,$ and $f'(\i) \ge \frac 12 \gamma(E).$  From \cite[Proposition 6.5]{Tol14}, there exists a Borel function $w(z)$ on $E$ with $0 \le |w(z)| \le 1$ such that $f(z) = \CT (w\mathcal H^1)(z)$ for $z\in W.$
For $\lambda\in E,$ by Lemma \ref{RhoGContLemma}, we see that $\rho (\mathcal I (f))$ is $\gamma$-continuous at $\lambda.$  
Using \eqref{IForHILemmaEq1} and Theorem \ref{GPTheorem1}, we infer that 
\begin{eqnarray}\label{IForHILemma2Eq1}
\ v^+(w\mathcal H^1, \Gamma) (z) = v^-(w\mathcal H^1, \Gamma) (z) = \rho (\mathcal I (f))(z), ~\gamma|_E-a.a.,
\end{eqnarray}
which implies $w(z) = 0,~\mathcal H^1|_E-a.a..$ This is a contradiction. The claim is proved.

From the above claim, we only need to prove $\gamma(\partial U \cap \mathcal R_0) = 0.$ Suppose that $\gamma(\partial U \cap \mathcal R_0) > 0.$
By Lemma \ref{CTMaxFunctFinite}, we find a compact subset $E \subset \partial U \cap \mathcal R_0$ such that $\gamma (E) > 0,$
\[
\ |{\mathcal C}_\epsilon (g_j\mu) (\lambda)|, ~ \mathcal M_{g_j\mu} (\lambda) \le M_{j} < \infty, ~ \lambda \in E \text{ for } j \ge 1.
\]
From Proposition \ref{GammaPlusThm}, there exists $\eta\in M_0^+(E)$ such that $\eta$ is of $1$-linear growth, $\|\CT_\epsilon\eta \| \le 1,$ and $\gamma(E) \lesssim \|\eta\|.$ Using Lemma \ref{GammaExist} and the above claim, we infer that $\gamma(\mathcal {ND}(\eta)) = 0.$ 

Now applying Lemma \ref{distributionLemma}, for given $j \ge 1,$ then there are two functions $F_1\in L^\infty (\mu)$ and $F_2\in L^\infty (\eta)$ with 
\begin{eqnarray}\label{IForHIthmEq00}
\ F_1(z) = \mathcal C(\eta)(z),~ |g_j|\mu _{\C \setminus \mathcal X_\eta}-a.a., 	
\end{eqnarray}
where $\mathcal X_\eta$ is defined before Lemma \ref{distributionLemma}, and 
\begin{eqnarray}\label{IForHIthmEq01}
\ F_2(z) = \mathcal C(g_j\mu)(z),~ \eta-a.a.	
\end{eqnarray}
such that in the sense of distribution,
 \begin{eqnarray}\label{IForHIthmEq1}
 \ \bar\partial (\mathcal C (\eta)\mathcal C (g_j\mu)) = - \pi(F_1g_j\mu + F_2 \eta).
 \end{eqnarray}
 From \eqref{IForHILemmaEq0}, we have  
 \[
 \ \mathcal C \eta(z)\mathcal C (g_j\mu)(z) = \mathcal C (\mathcal I(\mathcal C \eta)g_j\mu)(z), ~ \gamma |_{\C\setminus E}-a.a., 
 \]
 which implies $U \subset D :=\{\mathcal C \eta (z) \mathcal C (g_j\mu)(z) = \mathcal C (\mathcal I(\mathcal C \eta)g_j\mu)(z)\}~ \gamma-a.a..$
 By Lemma \ref{CauchyTLemma}, $\mathcal C \eta,$ $\mathcal C (g_j\mu),$ and $\mathcal C (\mathcal I(\mathcal C \eta)g_j\mu)$ are $\gamma$-continuous $\gamma |_{\mathcal R_0}-a.a..$ So $\mathcal C \eta\mathcal C (g_j\mu)$ is $\gamma$-continuous $\gamma |_{\mathcal R_0}-a.a.$. For $\lambda\in \partial U\cap \mathcal R_0, $ we have $\underset{\delta\rightarrow 0}{\overline \lim}\frac{\gamma (\D(\lambda, \delta) \cap U)}{\delta} > 0$ since $U$ is connected. Set $A_\epsilon =\{|\mathcal C \eta(z)\mathcal C (g_j\mu)(z) - \mathcal C \eta(\lambda)\mathcal C (g_j\mu)(\lambda)| \le \epsilon\}$ and $B_\epsilon =\{|\mathcal C (\mathcal I(\mathcal C \eta)g_j\mu)(z) - \mathcal C (\mathcal I(\mathcal C \eta)g_j\mu)(\lambda)| \le \epsilon\}.$ Since
 \[
 \ \D(\lambda, \delta) \cap U \subset (\D(\lambda, \delta) \cap U \cap A_\epsilon \cap B_\epsilon \cap D)\cup A_\epsilon^c \cup B_\epsilon^c, \gamma-a.a.,
 \]
 by Theorem \ref{TolsaTheorem} (2), we get
 \[
 \ \underset{\delta\rightarrow 0}{\overline \lim}\frac{\gamma (\D(\lambda, \delta) \cap U \cap A_\epsilon \cap B_\epsilon \cap D)}{\delta} \ge \dfrac{1}{C_T}\underset{\delta\rightarrow 0}{\overline \lim}\frac{\gamma (\D(\lambda, \delta) \cap U)}{\delta} > 0. 
 \]
 Thus, there exists $\{\lambda_n\}\subset U$ with $\lambda_n\rightarrow \lambda$ such that $\mathcal C \eta(\lambda_n)\mathcal C (g_j\mu)(\lambda_n) \rightarrow \mathcal C \eta(\lambda)\mathcal C (g_j\mu)(\lambda),$ $\mathcal C (\mathcal I(\mathcal C \eta)g_j\mu)(\lambda_n) \rightarrow \mathcal C (\mathcal I(\mathcal C \eta)g_j\mu)(\lambda),$ and $\mathcal C \eta(\lambda_n)\mathcal C (g_j\mu)(\lambda_n) = \mathcal C (\mathcal I(\mathcal C \eta)g_j\mu)(\lambda_n).$ Hence,
 \[
 \ \mathcal C \eta(z)\mathcal C (g_j\mu)(z) = \mathcal C (\mathcal I(\mathcal C \eta)g_j\mu)(z), ~ \gamma-a.a.. 
 \]
 Using Lemma \ref{RhoExistLemma}, we see that there exists $\mathcal Q$ with $\gamma(\mathcal Q)=0$ such that 
 \begin{eqnarray}\label{IForHIthmEq3}
 \ \mathcal I(\mathcal C \eta)(z) = \mathcal C \eta(z),~ \mu _{\mathcal R_0\setminus \mathcal Q}-a.a..
 \end{eqnarray}
 Since a $\gamma$ zero set is also a $\area$ zero set, we have
\[
\ \mathcal C \eta(z)\mathcal C (g_j\mu)(z) = \mathcal C (\mathcal I(\mathcal C \eta)g_j\mu)(z), ~ \area-a.a..
\]
Using \eqref{IForHIthmEq1} and \eqref{CTDistributionEq}, we have
\[
\ \mathcal I(\mathcal C \eta)(z)g_j(z)\mu = F_1(z)g_j(z)\mu + F_2(z) \eta.
\]
From \eqref{IForHIthmEq00} and \eqref{IForHIthmEq3}, we have
\[
\ F_1(z) g_j(z) = \mathcal I(\mathcal C \eta)(z) g_j(z) = \mathcal C \eta(z)g_j(z), ~\mu _{E\setminus (\mathcal Q\cup \mathcal X_\eta)}-a.a..
\]
Therefore, together with \eqref{IForHIthmEq01}, we get
\[
\ F_2(z) = \mathcal C(g_j\mu)(z) = 0,~\eta-a.a.
\]
since $\eta(\mathcal Q\cup \mathcal X_\eta) = 0$. This contradicts  $\eta(E) > 0.$  
This completes the proof.
 \end{proof}
 
 \begin{proof}(Theorem \ref{thmA} (2)$\Rightarrow$(1)) By Theorem \ref{thmA} (2) (d), we see that the map $\mathcal I_i = \rho_i^{-1}$ is an algebraic homomorphism from $H^\infty(U_i)$ to $R^{t, \i}(K_i, \mu _{\Delta_i})$ for $i\ge 1.$
Using Lemma \ref{IForHILemma}, we conclude that 
\[
\ \partial U_i \subset \mathcal F_{\Delta_i}, ~\gamma-a.a.,\text{ for } i \ge 1.
\]
For a given $i \ge 1$ and $\lambda \in \partial U_i,$ since $U_i$ is connected, 
$\underset{\delta\rightarrow 0}{\overline \lim}\frac{\gamma (\D(\lambda, \delta) \cap U_i)}{\delta} > 0.$
Hence, by Proposition \ref{RTFRProp2},
\[
\ \partial U_i \subset \partial_1 K_j \subset \mathcal F_{\Delta_j}, ~\gamma-a.a.,\text{ for } j \ne i.
\]
Thus, by Proposition \ref{RTFRProp1} (3), we have
\[
\ \partial U_i \subset \bigcap_{j = 0}^\i \mathcal F_{\Delta_j} \approx \mathcal F , ~\gamma-a.a..
\]
The theorem is proved. 
\end{proof}

\begin{corollary}\label{thmAB}
 Let $K$ be a compact set, $1\le t < \infty,$ and $\mu\in M_0^+(K).$ If $\partial K \subset \mathcal F,~ \gamma-a.a.,$ then there is a Borel partition $\{\Delta_i\}_{i=0}^\infty$ of $\mbox{spt}(\mu)$ and compact subsets $\{K_i\}_{i=1}^\infty$ such that $\Delta_i \subset K_i$ for $i \ge 1$,
 \[
 \ R^t(K, \mu ) = L^t(\mu _{\Delta_0})\oplus \bigoplus _{i = 1}^\infty R^t(K_i, \mu _{\Delta_i}),
 \]
and the following statements are true:

(a) If $i \ge 1$, then $S_{\mu _{\Delta_i}}$ on $R^t(K_i, \mu _{\Delta_i})$ is irreducible.

(b) If  $i \ge 1$ and $U_i :=\mbox{abpe}( R^t(K_i, \mu _{\Delta_i}))$, then $U_i$ is connected and $K_i = \overline{U_i}$.

(c) If  $i \ge 1$, then the evaluation map $\rho_i: f\rightarrow f|_{U_i}$ is an isometric isomorphism and a weak-star homeomorphism from $R^{t, \i}(K_i, \mu _{\Delta_i})$ onto $H^\infty(U_i)$.
\end{corollary}

\begin{proof}
Let $\text{abpe}(R^t(K, \mu)) = \cup_{i=1}^\i U_i,$ where $U_i$ is a connected component.
 By the assumption, we have $\partial U_i \cap \partial  K \subset \mathcal F,~\gamma-a.a.$ for all $i \ge 1.$ 
 Therefore, by Theorem \ref{thmA}, we see that the decomposition in Theorem \ref{thmA} holds. We only need to show that 
 \[
 \ R^t(K, \mu _{\Delta_0}) = L^t(\mu _{\Delta_0}).
 \]	
 In fact, from \eqref{ABPETheoremRTEq1}, 
 \[
 \ \text{int}(K) \setminus \mathcal F_{\Delta_0} \approx \text{abpe}(R^t(K, \mu _{\Delta_0})) = \emptyset,~ \gamma-a.a..
 \]
 Hence, $\mathcal F_{\Delta_0} = \C,~\gamma-a.a.$ since $\partial  K \subset \mathcal F \subset \mathcal F_{\Delta_0},~\gamma-a.a..$ The proof follows from \eqref{acZero}.	
\end{proof}

As an application of Corollary \ref{thmAB} and Proposition \ref{RTFRProp2}, we have the following corollary which extends the results of \cite{thomson}, \cite{ce93}, and \cite{b08}.

\begin{corollary}\label{thmB}
 Let $K$ be a compact set such that $\gamma(\partial K \setminus \partial_1 K) = 0$.
 Suppose that $1\le t < \infty$ and $\mu\in M_0^+(K).$ 
Then there is a Borel partition $\{\Delta_i\}_{i=0}^\infty$ of $\mbox{spt}(\mu)$ and compact subsets $\{K_i\}_{i=1}^\infty$ such that $\Delta_i \subset K_i$ for $i \ge 1$,
 \[
 \ R^t(K, \mu ) = L^t(\mu _{\Delta_0})\oplus \bigoplus _{i = 1}^\infty R^t(K_i, \mu _{\Delta_i}),
 \]
and the following statements are true:

(a) If $i \ge 1$, then $S_{\mu _{\Delta_i}}$ on $R^t(K_i, \mu _{\Delta_i})$ is irreducible.

(b) If  $i \ge 1$ and $U_i :=\mbox{abpe}( R^t(K_i, \mu _{\Delta_i}))$, then $U_i$ is connected and $K_i = \overline{U_i}$.

(c) If  $i \ge 1$, then the evaluation map $\rho_i: f\rightarrow f|_{U_i}$ is an isometric isomorphism and a weak-star homeomorphism from $R^{t, \i}(K_i, \mu _{\Delta_i})$ onto $H^\infty(U_i)$.
\end{corollary}

\bigskip

{\bf Acknowledgments.} 
The authors would like to thank Professor John M\raise.45ex\hbox{c}Carthy for carefully reading through the manuscript and providing many useful comments.
\bigskip

\bibliographystyle{amsplain}

\end{document}